\documentclass[11pt]{article}

\usepackage{amssymb,amsmath,amsfonts}
\usepackage{graphicx,color,enumitem}
\usepackage{amsthm} 
\usepackage{bm}
\usepackage[round]{natbib}
\usepackage{geometry}

\usepackage{bbm}
\usepackage{capt-of}

\newcommand{\E}{{\mathbb E}}

\renewcommand{\P}{{\mathbb P}}

\newcommand{\R}{{\mathbb R}}
\renewcommand{\S}{{\mathbb S}}

\newcommand{\Fcal}{{\mathcal F}}

\newcommand{\Mid}{{\ \Big|\ }}

\newcommand{\Fc}{{\mathcal F}}

\newcommand{\id}{{\rm id}}

\DeclareMathOperator{\tr}{tr}
\DeclareMathOperator{\Tr}{Tr}
\DeclareMathOperator{\vecop}{vec}

\DeclareMathOperator{\supp}{supp}

\newtheorem{theorem}{Theorem}

\newtheorem{definition}[theorem]{Definition}

\newtheorem{lemma}[theorem]{Lemma}

\newtheorem{proposition}[theorem]{Proposition}
\newtheorem{remark}[theorem]{Remark}

\theoremstyle{definition}
\newtheorem{example}[theorem]{Example}

\numberwithin{equation}{section}
\numberwithin{theorem}{section}

\definecolor{darkgreen}{rgb}{0,0.7,0}

\DeclareMathOperator{\diag}{diag}

\newcommand{\iii}{{\vert\kern-0.25ex\vert\kern-0.25ex\vert}}

\usepackage{mathtools}
\mathtoolsset{showonlyrefs}

\usepackage[colorinlistoftodos, textwidth=4cm, shadow]{todonotes}
\RequirePackage[colorlinks,citecolor=blue,urlcolor=blue]{hyperref}

\begin{document}

\title{The Laplace transform of the integrated Volterra Wishart process}

\author{Eduardo Abi Jaber\thanks{Universit\'e Paris 1 Panth\'eon-Sorbonne, Centre d'Economie de la Sorbonne, 106, Boulevard de l'H\^opital, 75013 Paris, eduardo.abi-jaber@univ-paris1.fr. My  work was supported by grants from R\'egion Ile-de-France. I would like to thank Aur\'elien Alfonsi,  Martin Larsson,  Sergio Pulido and Mathieu Rosenbaum for helpful comments and fruitful discussions.}}

\maketitle

\begin{abstract}
	We establish an explicit expression for the  conditional Laplace transform of the integrated Volterra Wishart process in terms of a certain resolvent of the covariance function. 
	The core ingredient is the derivation of the conditional Laplace transform of general Gaussian processes  in terms of Fredholm's determinant and resolvent. 
	Furthermore, we link the characteristic exponents to a system of non-standard infinite dimensional matrix Riccati equations. This leads to a second representation of the Laplace transform for a special case of convolution kernel.  In practice, we show that both representations can be  approximated by either closed form solutions of conventional Wishart distributions or finite dimensional matrix Riccati equations stemming from conventional linear-quadratic models. This allows fast pricing in  a variety of highly flexible models, ranging from bond pricing in  quadratic short rate models  with rich autocorrelation structures, long range dependence  and possible default risk, to pricing basket options with covariance risk in multivariate rough volatility models.  	\\[2ex] 
\end{abstract}


\newpage
\section{Introduction}

We are interested in the $d\times d$ Volterra Wishart process $XX^\top$ where $X$ is the $d\times m$-matrix valued Volterra Gaussian process
\begin{align}
X_t &= g_0(t) + \int_0^t K(t,s)  dW_s, 
\end{align}
for some given input curve $g_0:[0,T]\to {\R}^{d\times m}$, suitable kernel $K:[0,T]^2 \to \R^{d\times d}$ and $d\times m$-matrix Brownian motion $W$, for a fixed time horizon $T>0$.

The introduction of the kernel $K$ allows for flexibility in financial modeling as illustrated in the two following examples. 
First, one can consider asymmetric (possibly negative)  quadratic short rates  of the form 
$$r_t=\tr\left( X_t^\top Q X_t   \right) + \xi(t)$$
where $Q \in \S^d_+$, $\xi$ is an input curve used for matching market term structures and $\tr$ stands for the trace operator. The kernel $K$  allows for richer autocorrelation structures than the one generated with the  conventional  \citet{hull1990pricing} and  \citet*{cox2005theory}  models.   
Second, for $d=m$, one can build stochastic covariance models for $d$--assets $S=(S^1,\ldots, S^d)$ by considering the following dynamics for the stock prices:
\begin{align*}
dS_t &= \diag(S_t) X_t dB_t 
\end{align*}
where $B$ is $d$-dimensional and correlated with $W$. Then, the instantaneous covariance between the assets is stochastic and given by $\frac{d\langle \log S \rangle_t}{dt} = X_tX_t^\top \in \S^d_+$.  When $d=m=1$, one recovers the  Volterra version of the \citet{stein1991stock} or \citet{schobel1999stochastic} model. 
Here, singular kernels $K$ satisfying $\lim_{s\uparrow t}|K(t,s)|=\infty$, allow to take into account roughness of the sample paths of the volatility, as documented in \citet{bennedsen2016decoupling,gatheral2018volatility}. As an illustrative example for $d=m=1$, one could  consider the Riemann-Liouville fractional Brownian motion 
\begin{align}\label{eq:RLfbm}
X_t = \frac{1}{\Gamma(H+1/2)}  \int_0^t (t-s)^{H-1/2}dW_s,
\end{align}
either  with $H \in (0,1/2)$ to reproduce roughness when modeling the variance process, or with $H \in (1/2,1)$ to account for long memory in short rate models.

In both cases, integrated quantities of the form $\int_0^{\cdot} X_s X_s^\top ds $ play a key role for  pricing zero-coupon bonds and options on covariance risk. In order to keep the model tractable, one needs to come up with fast pricing and calibration techniques. The main objective of the paper is to show that these models remain highly tractable, despite the inherent non-markovianity and non-semimartingality due to the introduction of the kernel $K$.  For $w\in \S^d_+$,
our main result (Theorem~\ref{T:charXX}) provides the explicit expression for the conditional  Laplace transform:
\begin{align}
\E\left[ \exp\left(-\int_t^T \tr\left( w X_sX_s^\top\right)ds \right)\Mid \Fc_t \right]=\exp\left(\phi_{t,T}  +  \int_{(t,T]^2} \!\tr\left(  g_t(s)^\top \Psi_{t,T} (ds,du) g_t(u)\right)\right),\;\;\;\;
\end{align}
where $(\phi,\Psi)$ are defined by 
\begin{align}
\partial_t\phi_{t,T} &=   -  m\int_{(t,T]^2} \tr\left(  \Psi_{t,T}(ds,du) K(u,t) K(s,t)^\top  \right), \quad 	\phi_{T,T}= 0,    \\
\Psi_{t,T}(ds,du) &= -w  \delta_{\{s=u\}} (ds,du) - \sqrt w R^w_{t,T}(s,u) \sqrt w  ds du,  
\end{align}
with $g_t(s)=\E[X_s|\Fc_t]$ the forward process, $C_t(s,u)=\E[(X_s-g_t(s))(X_u-g_t(u))^\top|\Fc_t]$ the conditional covariance function, and  $R^w_{t,T}:[0,T]^2 \to \R^{d\times d}$ the Fredholm resolvent of $(-2\sqrt{w} C_t\sqrt w )$ on $[0,T]$ given by
$$ R^w_{t,T}(s,u)=-2\sqrt{w} C_t(s,u)\sqrt w- \int_t^T 2\sqrt{w} C_t(s,z)\sqrt w R^w_{t,T}(z,u)dz, \quad t\leq s,u\leq T. $$

{Using the integral operator  $\bold{C}_t$ induced by the covariance kernel $C_t$, i.e.~$(\bold{C}_tf)(s)=\int_0^T C_t(s,u)f(u)du$ for $f\in L^2\left([0,T],\R^{d\times m}\right)$, the Laplace transform can be re-expressed in analytic form 
	\begin{align}
	\E\left[ \exp\left(-\int_t^T \tr\left( w X_sX_s^\top\right)ds \right)\Mid \Fc_t \right]=\frac{\exp\left( -\langle g_t,   \sqrt w \left( \id + 2\sqrt w \bold{C}_t \sqrt w  \right)^{-1} \sqrt w g_t \rangle_{L^2_t}\right)}{\det\left( \id + 2\sqrt w \bold{C}_t \sqrt w  \right)^{m/2}}
	\end{align}
	where $\langle f,g \rangle_{L^2_t}=\int_t^T \tr\left(f(s)^\top g(s)\right)ds$ and $\det$ stands for the Fredholm determinant.}

{The Laplace transform is exponentially quadratic  in the forward process $(g_t)_{t\leq T}$, and cannot in general be recovered from that of finite dimensional affine Volterra processes introduced in \citet{AJLP17}, see Remark~\ref{R:dynamic}. We also mention that the models studied here are quadratic constructions of Gaussian processes and do not pose any difficulty regarding existence and uniqueness, in contrast for instance with conventional Wishart processes that go beyond squares of Gaussians, see \citet{bru1991wishart}.}

{Furthermore, we link $\Psi$ to a system of non-standard infinite dimensional backward Riccati equations in the general case of non-convolution kernels.   This allows us to deduce  a second representation of the Laplace transform for a special case of convolution kernels in the form 
 $$ 
 K(t,s)=k(t-s)\bm 1_{s \leq t}  \quad \mbox{ such that } \quad  k(t) = \int_{\R_+} e^{-xt} \mu(dx), \quad t > 0,
 $$
for some suitable signed measure $\mu$, showing, similarly to \citet{cuchiero2019markovian,harms2019affine}, that the Volterra Wishart process can be seen as a superposition of possibly infinitely many conventional linear-quadratic models written on the infinite dimensional process
  $$  Y_t(x) = \int_0^t e^{-x(t-s)} dW_s, \quad  t\geq 0, \quad x \in \R_+.$$ 
 In particular, this second representation not only allows us to recover the 
expressions for the Laplace transform derived in the aforementioned articles but, most importantly, provides an explicit solution for the corresponding infinite dimensional Riccati equations. 
}

Although explicit, the expression for the Laplace transform is not known in closed form, except for certain cases. We provide two approximation procedures  either by closed form solutions   of conventional Wishart distributions (Section~\ref{S:approx1}) or finite dimensional matrix Riccati equations  stemming from conventional linear-quadratic models (Section~\ref{S:approx2}).  These approximations can then be used to  price  bonds with possible default risk, or options on covariance in multivariate (rough) volatility models by Laplace transform techniques (Section~\ref{S:applications}).  
\vspace{0.2cm}

\textbf{Literature}
Conventional Wishart  processes initiated by  \citet{bru1991wishart} and  introduced in finance by \citet{gourieroux2003wishart} have been intensively applied, together with their variants, in  term structure  and stochastic covariance modeling, see for instance 
\citet{A15,buraschi2010correlation, CFMT:11, cuchiero2016affine, da2007option,da2008multifactor,gourieroux2009wishart, muhle2012option}. Conventional linear quadratic models have been characterized in  \citet{chen2004quadratic,cheng2007linear}.  
Volterra Wishart processes have been recently studied in  \citet{cuchiero2019markovian,yue2018fractional}. Applications of certain quadratic Gaussian processes can be found in \citet{benth2018non, corcuera2013short,harms2019affine, kleptsyna2002new}. Gaussian stochastic volatility models have  been treated in \citet{gulisashvili2018large, gulisashvili2019extreme}.

\vspace{0.2cm}
\textbf{Outline}  In Section~\ref{S:quadratic}  we derive the Laplace transform of  general quadratic Gaussian processes in $\R^N$, we provide a first approximation procedure by closed form expressions and link the characteristic exponent to non-standard Riccati equations. These results are then used in Section~\ref{S:Wishart} to deduce the Laplace transforms of Volterra Wishart processes. We also provide a second representation formula for the Laplace transform together with an approximation scheme for a special class of convolution kernels. Section~\ref{S:applications} presents applications to pricing: (i) bonds  in quadratic Volterra short rate models with possible default risk; (ii) options on volatility for basket products in Volterra Wishart  (rough) covariance models. Some technical results are collected in the appendices.

\vspace{0.2cm}
\textbf{Notations} For $T>0$, we define $L^2([0,T]^2,\R^{N\times N})$ to be the space of  measurable functions $F:[0,T]^2 \to \R^{N\times N}$ such that 
\begin{align*}
\int_0^T \int_0^T |F(t,s)|^2 dt ds < \infty.
\end{align*}
For any $F,G \in  L^2([0,T]^2,\R^{N\times N})$ we define the $\star$-product   by
\begin{align}\label{eq:starproduct}
(F \star G)(s,u) = \int_0^T F(s,z) G(z,u)dz, \quad  s,u \leq T,
\end{align}
which is well-defined in $L^2([0,T]^2,\R^{N\times N})$ due to the Cauchy-Schwarz inequality.  We denote by $F^*$ the adjoint kernel of $F$ in $L^2([0,T],\R^{N\times N})$, that is 
$$  F^*(s,u)=F(u,s)^\top, \quad  s,u \leq T.$$
 For any  kernel $F \in L^2([0,T]^2,\R^{N\times N})$, we denote by  $\bold F$ the integral operator from $L^2([0,T],\R^N)$ into itself induced by the kernel $F$ that is 
\begin{align}
 ({\bold F} g)(s)=\int_0^T F(s,u) g(u)du,\quad g \in L^2([0,T],\R^N).
\end{align}
  If $\bold{F}$ and $\bold{G}$ are two integral operators induced by the kernels $F$ and $G$  in $L^2([0,T]^2,\R^{N\times N})$, then $\bold{F}\bold{G}$ is an integral operator induced by the kernel $F\star G$.
 
 $\S^N_+$ stands for the cone of symmetric non-negative semidefinite  $N\times N$-matrices, $\tr$ denotes the trace of a matrix and $I_N$ is the $N\times N$ identity matrix.  The vectorization operator is denoted by $\vecop$ and the Kronecker product by $\otimes$, we refer to Appendix~\ref{A:matrix} for more details.

\section{Quadratic Gaussian processes}\label{S:quadratic}
Throughout this section, we fix  $T>0$, $N\geq 1$ and let $Z$ denote a $\R^N$-valued square-integrable Gaussian process on a filtered probability space $(\Omega, \Fc, (\Fc_t)_{t\leq T},\P)$ with mean function $g_0(s)=\E[Z_s]$ and covariance kernel given by  $C_0(s,u)=\E[(Z_s-g_0(s))(Z_u-g_0(u))^\top]$, for each $s,u \in [0,T]$. We note that $g_0$ and $C_0$ may depend on $T$, but we do not make this dependence explicit to ease notations.

 \subsection{Fredholm's representation and first properties}
Assume that $C_0$ is continuous in both variables.  Then, there exists a kernel $K_T \in L^2([0,T]^2,\R^{N\times N})$ and a $N$-dimensional Brownian motion $W$ such that 
\begin{align}\label{eq:Zrep}
Z_t = g_0(t) + \int_0^T K_T(t,s)dW_s,
\end{align}
for all $t\leq T$, see \citet[Theorem 12 and Example 2]{sottinen2016stochastic}. In particular, $C_0=K_T \star K_T^*$, that is 
$$ C_0(s,u)=\int_0^T K_T(s,z)K_T(u,z)^\top dz, \quad s,u\leq T. $$

For any $t\leq s$, $Z_s$ admits the following decomposition 
\begin{align}\label{eq:Zs}
Z_s = g_0(s) + \int_0^t K_T(s,u)dW_u + \int_t^T K_T(s,u)dW_u,
\end{align}
showing that conditional on $\Fc_t$, $Z_s$ is again a Gaussian process with conditional mean 
\begin{align*}
g_{t}(s)=\E[Z_s|\Fc_t]= g_0(s) + \int_0^t K_T(s,u)dW_u , \quad t \leq s\leq T, 
\end{align*}
and conditional covariance function 
\begin{align}
C_{t}(s,u)&=\E[(Z_s-g_t(s))(Z_u-g_t(u))^\top|\Fc_t] \nonumber\\
&= \int_t^T K_T(s,z)K_T(u,z)^\top dz, \quad\quad \quad  t\leq s,u \leq T. \label{eq:Ct}
\end{align}
Again we drop the possible dependence of $g_t$ and $C_t$ on $T$, and we note in particular that for each $s,u\leq T $, $t\to C_{t}(s,u)$ is absolutely continuous  on $[0,s\wedge u]$ with density 
\begin{align}\label{eq:Ctdensity}
\dot C_{t}(s,u)=-K(s,t)K(u,t)^\top,
\end{align}
and  that  the process $t \mapsto g_{t}(s)$ is a semimartingale on $[0,s)$ with dynamics
$$ d g_{t}(s)= K_T(s,t) dW_t, \quad t < s .$$

 We are chiefly interested in the $\S^N_+$-valued process  $ZZ^\top$.
 The following remark shows that, in general, $ZZ^\top$ cannot be recast as an affine Volterra process as studied in \citet{AJLP17}. 
 
 \begin{remark} \label{R:dynamic}
 	To fix ideas, we set $ g_0\equiv Z_0 \in \R^{N}$. An application of It\^o's formula yields  
 	\begin{align*}
 	g_{t}(s)g_{t}(s)^\top &= Z_0Z_0^\top + \int_0^t K_T(s,u) K_T(s,u)^\top du \\
 	&\quad + \int_0^t K_T(s,u) dW_u {g_{u}(s)}^\top + \int_0^t {g_{u}(s)} dW_u^\top  K_T(s,u)^\top, \quad t<s.
 	\end{align*}    
 	Taking the limit $s \to t$ leads to the  dynamics 
 	\begin{align}
 	Z_tZ_t^\top &= Z_0 Z_0^\top + \int_0^t  K_T(t,u)K_T(t,u)^\top du \nonumber \\
 	&\quad + \int_0^t K_T(t,u)dW_u {g_{u}(t)}^\top +\int_0^t {g_{u}(t)}dW_u^\top K_T(t,u)^\top. \label{eq:dynamics}
 	\end{align}
 	{This shows, that in general, because of the presence of the infinite dimensional process $t\mapsto g_t$ in the dynamics,  $ZZ^\top$ does not satisfy a stochastic Volterra equation in the form 
 	$$  Y_t =Y_0 +  \int_0^t K(t,s)b(Y_s)ds + \int_0^t \sigma(Y_s)dW_s^\top K(t,s)^\top +  \int_0^t K(t,s)dW_s\sigma(Y_s)^\top,$$
 	where $b,\sigma:\R^{N}\mapsto \R^{N \times N}$. For this reason, $ZZ^\top$ falls beyond the scope of the processes studied in \citet{AJLP17}. Except for very specific cases, for instance, when  $K_T \equiv I_N$, we have {$ g_{u}(s)=Z_u$ for all $u<s$}, and \eqref{eq:dynamics} reduces to the  well-known dynamics of Wishart processes as introduced by \cite{bru1991wishart}.}
 \end{remark}
 
Whence, the conditional Laplace transform of $ZZ^\top$ cannot be deduced from \citet[Theorem 4.3]{AJLP17}. Nonetheless, it can be directly computed from Wishart distributions  that we recall in Appendix~\ref{A:wishart}.  
 \begin{theorem}\label{T:charzz}   Fix $t \leq s \leq T$.
 	Conditional on $\Fc_t$, $Z_sZ_s^\top$ follows a Wishart distribution 
 	$$Z_s Z_s^\top \sim_{|\Fc_t} {\mbox{WIS}}_{N} \left( 1/2, g_{t}(s)g_{t}(s)^\top, 2 C_{t}(s,s) \right).$$	
 	Further, for any $u \in \S^N_+$, the conditional Laplace transform reads 
 	\begin{align}
 	\E\left[\exp\left(-  Z_s^\top u Z_s  \right) \Mid \Fc_t \right] = \frac {\exp\left(- g_{t}(s)^\top  u \left(I_N + 2C_{t}(s,s)u \right)^{-1}  g_{t}(s) \right)}{\det\left(I_N + 2 C_{t}(s,s)u\right)^{1/2}}.
 	\end{align}
 \end{theorem}
 
 \begin{proof}
 	Fix $t \leq s \leq T$, conditional on $\Fc_t$, it follows from \eqref{eq:Zs} that $Z_s$ is a Gaussian vector in $\R^N$ with mean vector $g_t(s) \in \R^N $ and covariance matrix $C_t(s,s)\in \R^{N\times N}$. The claimed result now follows from Proposition~\ref{P:Wishartchar}. 
 \end{proof}
 
 In particular, if $N=1$, $t=0$ and $s=T$, one obtains the well-known chi-square distribution 
 \begin{align*}
 \E\left[\exp\left(-  u Z^2_T  \right) \right] = \frac {\exp\left( \frac{-ug_0(T)^2} {1 + 2  u C_0(T,T)}  \right)}{\left(1 + 2  u C_0(T,T)\right)^{1/2}}, \quad u \geq 0.
 \end{align*}
 
The computation of the Laplace transform for the integrated squared process is more involved and is treated in the next subsection. 
 
\subsection{Conditional Laplace transform of the integrated quadratic process}
 We are interested in computing the conditional Laplace transform 
\begin{align}\label{eq:charZ}
\E\left[  \exp\left( -\int_t^T Z_s^\top w Z_s ds  \right)  \Mid \Fc_t\right], \quad w \in \S^N_+,  \quad t \leq T.
\end{align} 
For $t=0$ and for centered processes, such computations appeared several times in the literature showing that the quantity of interest can be decomposed as an infinite product of independent chi-square distributions, see for instance \citet{anderson1952asymptotic,cameron1959inversion, varberg1966convergence}.
The same methodology can be readily adapted to our dynamical case and makes use of the celebrated  Kac--Siegert/Karhunen--Lo\`eve representation of the process $Y=\sqrt{w}Z$ whose conditional covariance function is $C^w_{t} = \sqrt w C_{t} \sqrt w$, see \citet{kac1947theory, karhunen1946spektraltheorie, loeve1955probability}. For this,  we fix $t\leq T$, we consider the inner product on $L^2([t,T],\R^N)$ given by
$$  \langle f, g \rangle_{L^2_t}= \int_t^T f(s)^\top g(s) ds , \quad  \quad f,g \in L^2([t,T],\R^N),$$
and we assume that $ C_{t}$ is continuous in both variables\footnote{This is equivalent to assuming that the centered Gaussian process $(Z_{\cdot}-\E[Z_{\cdot}])$ is mean-square continuous.}. By definition, the covariance kernel  $C^w_{t}$ is symmetric and nonnegative in the sense that 
$$  C^w_{t}(s,u)=C^w_{t}(u,s)^\top, \quad s,t\leq T,$$
and 
$$   \int_t^T \int_t^T f(s)^\top C^w_{t}(s,u) f(u) du ds \geq 0, \quad f  \in L^2([t,T],\R^N).  $$
An application of Mercer's theorem, see  \citet[Theorem~1 p.208]{shorack2009empirical}, yields the existence of a countable orthonormal basis $(e^n_{t,T})_{n\geq 1}$ in $L^2([t,T],\R^N)$ and a sequence of nonnegative  real numbers $(\lambda^n_{t,T})_{n\geq 1}$ with $\sum_{n\geq 1} \lambda^n_{t,T} < \infty$  such that 
\begin{align}\label{eq:mercerC}
C^w_{t}(s,u)&= \sum_{n\geq 1} \lambda_{t,T}^n  e_{t,T}^n(s) e_{t,T}^n(u)^\top, \quad t\leq  s,u \leq T,
\end{align}
and
\begin{align}\label{eq:operatoreigen}
\int_t^T C^w_{t}(s,u) e^n_{t,T}(u) du = \lambda^n_{t,T}e_{t,T}^n(s), \quad t\leq s \leq T, \quad n \geq 1,
\end{align}
where the dependence of $(e_{t,T}^n,\lambda^n_{t,T})$  on $w$ is dropped to ease notations. This means that $(\lambda^n_{t,T},e^n_{t,T})_{n\geq 1}$ are the eigenvalues and the eigenvectors of the  integral operator $\sqrt{w}\bold{C}_{t} \sqrt w$ from $L^2([t,T],\R^N)$ into itself  induced by $C_{t}^{w}$:
\begin{align*}
(\sqrt{w} \bold{C}_{t} \sqrt{w} f)(s)=\int_t^T  C^w_{t}(s,u) f(u)du, \quad t\leq s \leq T, \quad f\in L^2([t,T],\R^N).
\end{align*}
 As a consequence of Mercer's theorem, conditional on $\Fc_t$, the  process $Y$ admits the   Kac--Siegert representation
\begin{align}\label{eq:Ys}
Y_s &=  \sqrt wg_{t}(s) + \sum_{n\geq 1}   \sqrt{\lambda_{t,T}^n} \xi^n e_{t,T}^n(s), \quad t\leq s\leq T,
\end{align}
where, conditional on $\Fc_t$, $(\xi_n)_{n\geq 1}$ is a sequence of independent standard Gaussian random variables, see \citet[Theorem~2 p.210 and the comment below (14) on p.212]{shorack2009empirical}.
We now introduce the quantities needed for the computation of \eqref{eq:charZ} in Theorem~\ref{T:charZZ} below. We denote by $\id$ the identity operator on $L^2([t,T],\R^N)$, i.e.~$(\id f)(s)=f(s)$, by $ (\id+2\sqrt{w}\bold{C}_{t} \sqrt{w})^{-1} $ the integral operator generated by the kernel 
\begin{align}\label{eq:kernelop2}
\sum_{n\geq 1} \frac 1 { 1+ 2\lambda_{t,T}^n} e_{t,T}^n(s) e_{t,T}^n(u)^\top,
\end{align} 
and we set 
\begin{align}\label{eq:detfredholm}
\det(\id+2\sqrt{w}\bold{C}_{t} \sqrt{w}):= \prod_{n\geq 1} \left({{1+2 \lambda_{t,T}^n}}\right).
\end{align}
The  last expression is well defined due to the convergence of the series $(\sum_{n=1}^m \lambda^n_{t,T})_{m\geq 1}$ and  the inequality 
$$  1 + 2 \sum_{n= 1}^m \lambda^n_{t,T}  \leq  \prod_{n=1}^m \left( 1 + 2 \lambda^n_{t,T} \right)   \leq \exp \left( 2\sum_{n= 1}^m \lambda^n_{t,T} \right),\quad m\geq 1. $$

\begin{theorem}\label{T:charZZ}
	Fix $w\in \S^N_+$ and $t\leq T$.
	Assume that  the function $(s,u)\mapsto C_t(s,u)$ is continuous. Then,
	\begin{align}\label{eq:charinfinite}
	\E\left[  \exp\left( \!-\!\int_t^T Z_s^\top w Z_s ds  \right)  \!\! \Mid \!\Fc_t \right]= \frac{\!\exp\left( -\langle g_{t},  \!\sqrt w\left( \id + 2 \sqrt{w}\bold{C}_{t} \!\sqrt{w} \right)^{-1} \!\!\sqrt w  g_{t} \rangle_{L^2_t} \right) }{\det\left( \id + 2\sqrt{w}\bold{C}_{t} \sqrt{w} \right)^{1/2}}. \quad \;\;\;
	\end{align}
\end{theorem}

\begin{proof}
	Fix $t\leq T$. Parseval's  identity gives $\langle  \sqrt w g_{t},  \sqrt w g_{t}  \rangle_{L^2_t} = \sum_{n\geq 1}   \langle  \sqrt w g_{t},   e^n_{t,T} \rangle_{L^2_t}^2 $ so that 
\begin{align*}
\int_t^T Z_s^\top w Z_s ds  = \langle  Y , Y \rangle_{L^2_t}=\sum_{n \geq 1}  \left( \sqrt{\lambda_{t,T}^n} \xi^n + \langle  \sqrt w g_{t}, e^n_{t,T}   \rangle_{L^2_t}     \right)^2,
\end{align*}
where the first equality follows from  the definition $Y:=\sqrt{w} Z$  and the second equality is a consequence of \eqref{eq:Ys}. 
By the independence of the sequence $(\xi^n)_{n\geq 1}$ and the dominated convergence theorem we can compute
\begin{align*}
	\E\left[  \exp\left( -\int_t^T Z_s^\top w Z_s ds  \right)   \Mid \Fc_t \right] &= \prod_{n\geq 1 } 	\E\left[  \exp\left( -\left( \sqrt{\lambda_{t,T}^n} \xi^n + \langle \sqrt w  g_{t},  e^n_{t,T}  \rangle_{L^2_t}   \right)^2 \right)\Mid \Fc_t \right] \\
	&= \prod_{n\geq 1 } \frac 1 {\sqrt{1+ 2\lambda_{t,T}^n}} \exp\left( -   \frac 1 {1+ 2\lambda_{t,T}^n} \langle \sqrt w  g_{t},  e^n_{t,T}  \rangle_{L^2_t}^2 \right) \\
	&= \det(\id+2\sqrt{w}\bold{C}_{t} \sqrt{w})^{-1/2} \\
	&\quad \quad \quad  \times \exp\left( -    \sum_{n\geq 1 } \frac 1 {1+ 2\lambda_{t,T}^n} \langle \sqrt w  g_{t},  e^n_{t,T}  \rangle_{L^2_t}^2 \right) ,
\end{align*}
where the second equality is obtained from the chi-square distribution, since the random variable  $\left( \lambda^n_{t,T} \xi^n + \langle \sqrt w  g_{t},  e^n_{t,T}  \rangle_{L^2_t}   \right)$ is Gaussian with mean $\langle \sqrt w  g_{t},  e^n_{t,T}  \rangle_{L^2_t}$ and variance $\lambda^n_{t,T}$, for each $n\geq 1$,  see  Proposition~\ref{P:Wishartchar}. The claimed expression now follows upon observing that, thanks to \eqref{eq:kernelop2},  
$$   \langle g_{t}  ,  \sqrt w\left( \id + 2 \sqrt{w}\bold{C}_t\sqrt{w} \right)^{-1} \sqrt w  g_{t} \rangle_{L^2_t}=  \sum_{n\geq 1 } \frac 1 {1+ 2\lambda_{t,T}^n} \langle \sqrt w  g_{t},  e^n_{t,T}  \rangle_{L^2_t}^2.$$
\end{proof}

\begin{remark}\label{R:Lidskii}
	The determinant \eqref{eq:detfredholm} is named after \citet{fredholm1903} who defined it for the first time through the following expansion 
	$$ \det(\id + \bold{C})= \sum_{n \geq 0} \frac 1 {n!} \int_t^T \ldots \int_t^T \det \left[ (C(s_i,s_j))_{1\leq i,j\leq n} \right] ds_1\ldots ds_n ,$$
	where  $\bold{C}$  is a generic integral operator of trace class with continuous kernel $C$.  Lidskii's theorem ensures that  Fredholm's definition is equivalent to  
	\begin{align*}
	\det(\id + \bold{C}) =\exp\left(\Tr\left( \log\left( \id + \bold{C} \right) \right)\right),
	\end{align*}  
	where $\Tr(\bold{C})=\int_t^T {\tr (C(s,s))}ds$, and consequently equivalent to the infinite product expression as in \eqref{eq:detfredholm}, refer to \citet{simon1977notes} for more details. 
\end{remark}

Closed form solutions are known in some standard cases.
\begin{example} Set $N=1$, $t=0$, $T=1$ and  $Z =W$, where $W$ is a standard 
	Brownian motion and $Z_0 \in \R$. Then, $g_0(s)=0$ and $C_0(s,u)=s\wedge u$ and the eigenvalues and eigenvectors of the eigenproblem  \eqref{eq:operatoreigen} are well-known and given by 
	$$ \lambda_{0,1}^n = \frac w {(n- 1/ 2)^2 \pi^2} \quad \mbox{and}  \quad e^n_{0,1}(s)= \sqrt 2 \sin \left( \left(n-\frac 1 2\right) \pi s \right),\quad n \geq 1.$$ 
	Using the identity 
	$$  \prod_{n\geq 1} \left({{1+2 \lambda_{0,1}^n}}\right) = \prod_{n\geq 1} \left({1+ \frac {2w} {(n- 1/ 2)^2 \pi^2} }\right)= \cosh{\sqrt{2w}},$$
	\eqref{eq:charinfinite} reads
\begin{align}\label{eq:cosh}
\E\left[  \exp\left( - w\int_0^1 W_s^2 ds  \right)  \right]= \left(\cosh{\sqrt{2w}}\right)^{-1/2}.
\end{align}
\end{example}

For arbitrary kernels $C$, the eigenpairs  $(\lambda^n,e^n)_{n\geq 1}$ are, in general, not known in closed form. This is the case for instance for the fractional Brownian motion. We provide in the next subsection an approximation by closed form formulas.

{\subsection{Approximation by closed form expressions}}\label{S:approx1}
A natural idea to approximate \eqref{eq:charinfinite} is to discretize the time-integral. Fix $t\leq T$ and let $(s_i^n,\alpha_i^n)$, $i=1,\ldots,n$ be a quadrature rule on $[t,T]$, i.e. 
$$  \int_t^T f(s)ds = \lim_{n\to \infty}\sum_{i=1}^n \alpha^n_i f(s_i^n).$$ By the dominated convergence theorem it follows that 
\begin{align*}
	\E\left[  \exp\left( -\int_t^T Z_s^\top w Z_s ds  \right)  \Mid \Fc_t \right] = \lim_{n\to \infty}  \E\left[ \exp\left( - \sum_{i=1}^n \alpha^n_i Z_{s^n_i}^\top w Z_{s^n_i}  \right)  \Mid \Fc_t \right],
\end{align*}
for all $w\in \S^N_+$. For each $n$, $(Z_{s_1^n},\ldots,  Z_{s_n^n})^\top$ being Gaussian, the right hand side is known in closed form. This is the object of the next proposition which will make use of the Kronecker product $\otimes$ and the vectorization operator $\vecop$, we refer to Appendix~\ref{A:matrix} for more details.
\begin{proposition}\label{eq:Papprox1}
		Fix $w\in \S^N_+$ and $t\leq T$.
	\begin{align}\label{eq:charapprox}
	\E\left[  \exp\left( -\int_t^T Z_s^\top w Z_s ds  \right)  \Mid \Fc_t \right]= \lim_{n \to \infty} \frac{\!\exp\left({-}g_{t}^{n\top}  w_n  \left( I_{nN} + 2  C^n_{t} w_n \right)^{-1} g^n_{t} \right)}{\det\left( I_{nN} + 2  C^n_{t}w_n\right)^{1/2}},
	\end{align}
 where  $w_n=\left( \diag(\alpha^n_1,\ldots,\alpha^n_n) \otimes w \right)  \in \R^{nN\times nN}$,  $g_{t}^n$ is the $nN$-vector 
	\begin{align}\label{eq:gn}
	g_{t}^n = \left( \begin{matrix}
	g_{t}(s^n_1) \\
	\vdots\\
	g_{t}(s^n_n) 
	\end{matrix}\right),
	\end{align}
	and $ C^n_{t}$ is the $nN\times nN$-matrix  with entries 
	\begin{align}\label{eq:cn}
(C_{t}^n)^{p,q} = C_{t}(s_i^n,s_k^n)^{jl},  \quad  p=(i-1)N+j, \; q=(k-1)N+l, 
	\end{align}
	for all $i,k=1,\ldots,n,$ and  $j,l=1,\ldots,N.$
\end{proposition}

\begin{proof}
	We simply observe that 
	$$ \sum_{i=1}^n \alpha_i^n Z_{s_i^n}^\top w   Z_{s_i^n} = \bold{Z}^{n\top} \left(\diag(\alpha^n_1,\ldots,\alpha^n_n) \otimes w \right) \bold{Z}^n,$$
	where $\bold{Z}^n=\vecop(Z^n)$ and $Z^n=(Z_{s^n_1},\ldots, Z_{s^n_n})$. Conditional on $\mathcal F_t$, $\bold{Z}^n$ being a Gaussian vector in $\R^{nN}$ with mean vector \eqref{eq:gn} and covariance matrix \eqref{eq:cn}, the claimed result readily follows from Proposition~\ref{P:Wishartchar} combined with the dominated convergence theorem. 
\end{proof}

We now illustrate the approximation procedure in practice for $N=1$. 
Consider a one dimensional fractional Brownian motion $W^H$ with Hurst index $H \in (0,1)$ and set 
\begin{align}\label{eq:Ih}
I(H)=\E\left[ \exp\left( -\int_0^1 \left(W^H_s\right)^2 ds \right) \right].
\end{align} 
The (unconditional) covariance function of the fractional Brownian motion is given by 
\begin{align}\label{eq:covfbm}
 C_0^H(s,u)= \frac 12 \left(|s|^{2H}+|u|^{2H}-|s-u|^{2H}\right).
\end{align} 
Fix $n \geq 1$ we consider two quadrature rules on $[0,1]$:  the left Riemann sum with $s_i^n=i/n$ and $\alpha_i^n=1/n$ and the Gauss-Legendre rule advocated in \citet{bornemann2010numerical}. Since $W^H$ is centered, \eqref{eq:gn}  reads $g^n_{0} =0 $ and  the right hand side in \eqref{eq:charapprox} reduces to
\begin{align}\label{eq:Ihn}
I^n(H)=\det \left( I_n +  2  C^{H,n}_0 \diag(\alpha^n_1,\ldots,\alpha^n_n) \right)^{-\frac 1 2},
\end{align}
where $C^{H,n}_0(i,j)= C^H_0(s^n_i,s^n_j) $, $i,j=1,\ldots,n$.
We proceed  as follows. First, we determine the reference value of \eqref{eq:Ih} for several values of $H$. For $H=1/2$, the exact value is  $I(1/2)=\cosh(\sqrt 2)^{-1/2}$, recall \eqref{eq:cosh}. For $H \in \{0.1,0.3,0.7,0.9\}$, we run  a Monte--Carlo simulation to estimate $I(H)$ with the trapezoidal rule with a 95\% confidence interval {and $10^6$ sample paths with $10^3$ time steps for each sample path.} Second,  for each value of $H$, we compute $I^n(H)$ as in \eqref{eq:Ihn}, for several values of $n$ with the left Riemann sum and the Gauss--Legendre quadrature. The results are collected in  Tables \ref{tablewishartapprox}--\ref{tablewishartapproxlegendre} and Figure \ref{fig:convergencewishart}  below.  We observe that the Gauss--Legendre quadrature performs better than the left Riemann sum rule, especially for higher values of $H$. When $H\geq 0.5$,  even with $n=10$, $I^n(H)$ with the Gauss--Legendre rule falls already within the 95\% confidence interval of the Monte--Carlo simulation.  Other quadrature rules can be used in Proposition~\ref{eq:Papprox1}, see for instance \citet{bornemann2010numerical}. {We refer to Remark~\ref{R:approx1} below for a numerical illustration in higher dimensions.}

\begin{center}
	\includegraphics[scale=0.62]{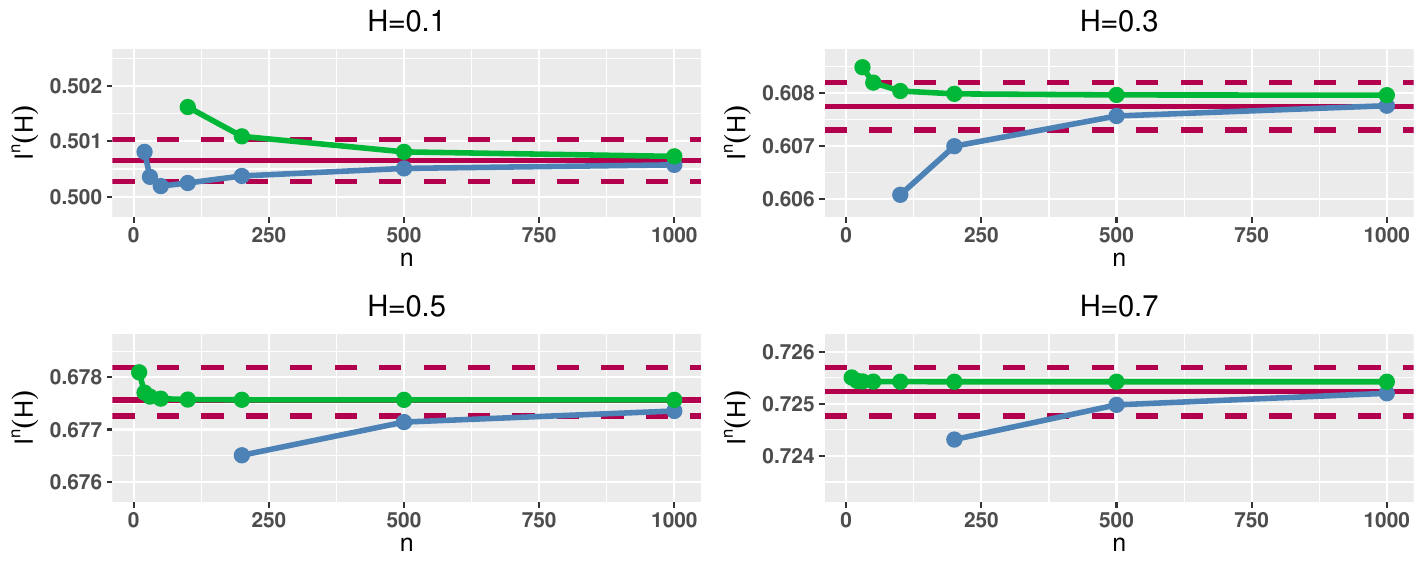}
	\rule{35em}{0.5pt}
	\captionof{figure}{{Convergence of $I^n(H)$ with the Riemann sum (blue) and the Gauss--Legendre quadrature (green)  towards the benchmark MC value $I(H)$ (red) for different values of $(H,n)$ from Table \ref{tablewishartapprox}. The dashed lines delimit the $90\%$ confidence interval of the Monte--Carlo simulation.}}
	\label{fig:convergencewishart}
\end{center}

\begin{table}[h!]
	\centering  
	\begin{tabular}{ccccccc } 
		\hline
		$H$   & 0.1 & 0.3 & 0.5 & 0.7 & 0.9 \\ 
		\hline \hline                 
		ref. $I(H)$ &   0.50065 & 0.60775 & 0.67757* & 0.72523 & 0.76023 \\ 
		
		&  &  &  &  &  \\ 
		$n \backslash I^n(H)$&  &  &  &  &  \\ 
		10 & 0.50310 & 0.59301 & 0.65763 & 0.70376 & 0.73779 \\ 
		20 & 0.50081 & 0.59961 & 0.66727 & 0.71445 & 0.74810 \\ 
		30 & 0.50027 & 0.60291 & 0.67160 & 0.71912 & 0.75386 \\ 
		50 & 0.50019 & 0.60433 & 0.67337 & 0.72101 & 0.75581 \\ 
		100 & 0.50025 & 0.60608 & 0.67545 & 0.72321 & 0.75801 \\ 
		200 & 0.50037 & 0.60701 & 0.67650 & 0.72431 & 0.75924 \\ 
		500 & 0.50051 & 0.60757 & 0.67714 & 0.72498 & 0.75992 \\ 
		1000 & 0.50058 & 0.60776 & 0.67735 & 0.72520 & 0.76015 \\ 
		\hline
	\end{tabular}
	\caption{Approximation of $I(H)$ by $I^n(H)$ with the left Riemann sum  for several values of $H$ with  $n$ ranging between $10$ and $1000$. *exact value for $I(1/2)$. }
	\label{tablewishartapprox} 
\end{table}

\begin{table}[h!]
	\centering  
	\begin{tabular}{ccccccc } 
		\hline
		$H$   & 0.1 & 0.3 & 0.5 & 0.7 & 0.9 \\ 
		\hline \hline                 
	ref. $I(H)$ &   0.50065 & 0.60775 & 0.67757* & 0.72523 & 0.76023 \\ 
		&  &  &  &  &  \\ 
		$n \backslash I^n(H)$&  &  &  &  &  \\ 
10 & 0.51331 & 0.61075 & 0.67810 & 0.72550 & 0.76039 \\ 
20 & 0.50665 & 0.60895 & 0.67771 & 0.72544 & 0.76038 \\ 
30 & 0.50447 & 0.60850 & 0.67763 & 0.72543 & 0.76038 \\ 
50 & 0.50279 & 0.60820 & 0.67759 & 0.72543 & 0.76038 \\ 
100 & 0.50162 & 0.60804 & 0.67758 & 0.72542 & 0.76038 \\ 
200 & 0.50109 & 0.60799 & 0.67757 & 0.72542 & 0.76038 \\ 
500 & 0.50081 & 0.60797 & 0.67757 & 0.72542 & 0.76038 \\ 
1000 & 0.50072 & 0.60797 & 0.67757 & 0.72542 & 0.76038 \\ 
		\hline
	\end{tabular}
	\caption{Approximation of $I(H)$ by $I^n(H)$ with the Gauss--Legendre quadrature for several values of $H$ with  $n$ ranging between $10$ and $1000$. *exact value for $I(1/2)$. }
	\label{tablewishartapproxlegendre} 
\end{table}


\subsection{Connection to Riccati equations}
The expression \eqref{eq:charinfinite} is reminiscent of the formula obtained for finite dimensional Wishart processes in \citet{bru1991wishart}  and more generally that of linear quadratic diffusions, see \citet{cheng2007linear}, suggesting a connection with infinite dimensional Riccati equations. 
Indeed, setting 
\begin{align*}
\phi_{t,T}&=-\frac 1 2 \Tr\left(\log \left( \id  + 2\sqrt{w}\bold{C}_{t} \sqrt{w}  \right)\right), \\
\bold{\Psi}_{t,T}&=-\sqrt w\left(  \id + 2\sqrt{w}\bold{C}_{t} \sqrt{w} \right)^{-1}\sqrt w,
\end{align*}
it follows from Remark~\ref{R:Lidskii} that \eqref{eq:charinfinite} can be rewritten as 
\begin{align}\label{eq:charinfinite2}
\E\left[  \exp\left( -\int_t^T Z_s^\top w Z_s ds  \right) \Mid \Fc_t \right]= \exp\left({\phi}_{t,T} +\langle g_{t},  \bold{\Psi}_{t,T}  g_{t} \rangle_{L^2_t} \right) , \quad t \leq {T}.
\end{align}
{Since  $t\to C_t(s,u)$ is absolutely continuous with density $\dot{C}_t(s,u)$ given by \eqref{eq:Ctdensity},  one would expect $t \to {\bold{C}_t} $ to be strongly  differentiable\footnote{
	We recall that  $t\mapsto \bold{C}_t$ is strongly differentiable  at time $t\geq  0$,  if there exists a bounded linear operator $\dot{\boldsymbol{C}}_t$  from $ L^2\left([0,T],\R^N\right)$  into  itself  such that 
	\begin{align}
	\lim_{h\to 0} \frac{1}{h} \| \boldsymbol{C}_{t+h}-\boldsymbol{C}_{t} -h \dot{\boldsymbol{C}}_{t} \|_{\rm{op}}=0, \quad \text{where }  \|\boldsymbol{G}\|_{\rm {op}}= \sup_{f \in L^2([0,T],\R^N)} \frac{\|\bold G f\|_{L^2}}{\|f\|_{L^2}}.
	\end{align}
}
 with  derivative $\dot{\bold{C}}_t$ given by the integral operator
 \vspace{-0.2cm}
\begin{align}\label{eq:diffC}
(\dot{\bold{C}}_t f)(s) = \int_t^T \dot C_t(s,u) f(u)du, \quad f \in L^2([0,T],\R^N), \quad s \leq T. 
\end{align}}
By taking the derivatives we get that $({\phi},\bold{\Psi})$ solves the following system of operator Riccati equations
\begin{align}
\dot{{\phi}}_{t,T} &=  \Tr \left( \bold{\Psi}_{t,T} 	\dot{\bold C}_{t} \right), && 	{{\phi}}_{T,T}=0, \label{eq:RiccatiopBigPhi} \\
\dot{\bold{\Psi}}_{t,T} &=  2  \bold{\Psi}_{t,T}	\sqrt{w}\dot{ \bold C}_{t}    \sqrt{w} \bold{\Psi}_{t,T},  &&  {{\bold{\Psi}}_{T,T}=-w \id}, \label{eq:RiccatiopPsi}
\end{align}
where $\dot{\bold{F}}_t$ denotes the derivative of $\bold{F}_t$ with respect to $t$.

This induces a system of Riccati equations for the kernels. To see this, we introduce the concept of resolvent.  Fix $t\leq T$ and define the kernel 
\begin{align}\label{eq:resolventdef}
R^{w}_{t,T}(s,u)=\sum_{n\geq 1} \left(\frac 1 { 1+ 2\lambda_{t,T}^n} -1\right)e_{t,T}^n(s) e_{t,T}^n(u)^\top, \quad t\leq s,u\leq T.
\end{align}
It is straightforward to check, using \eqref{eq:mercerC},  that for all $t \leq s,u\leq T$,
\begin{align}\label{eq:resequation}
\!\!\!2\int_t^T R^{w}_{t,T}(s,z)C_{t}^w(z,u) dz = \!2\int_t^T C_{t}^w(s,z) R^{w}_{t,T}(z,u) dz = -R^{w}_{t,T}(s,u) - 2C_t^w(s,u).\,\,\,\,\,\,
\end{align}
$R^w_{t,T}$ is called the resolvent kernel of $(-2C_{t}^w)$ and the integral operator $\bold{R}^w_{t,T}$ induced by  $R^w_{t,T}$  satisfies the relation
\begin{align}\label{eq:boldRboldC}
\bold{R}^w_{t,T} =(\id +2\sqrt{w}\bold{C}_{t} \sqrt{w})^{-1}- \id,
\end{align}
so that $\bold\Psi_{t,T}$ can be re-expressed in terms of the resolvent
\begin{align*}
\bold\Psi_{t,T} = - w \id - \sqrt{w}\bold{R}^w_{t,T}\sqrt{w}.
\end{align*}
The next theorem, whose proof  is postponed to Appendix~\ref{A:proof}, establishes the representation of the Laplace transform together with the  Riccati  equations \eqref{eq:RiccatiopBigPhi}-\eqref{eq:RiccatiopPsi} in terms of  the induced kernel 
\begin{align}\label{eq:psi_tT}
\Psi_{t,T}(ds,du)= - w\delta_{s=u} (ds,du) + \psi_{t,T}(s,u)ds du,
\end{align}
where $\psi_{t,T}=- \sqrt w R^w_{t,T}  \sqrt w  $ is the density of $\Psi_{t,T}$ with respect to the Lebesgue measure. We recall the $\star$-product defined in \eqref{eq:starproduct}. 

\begin{theorem}\label{T:char2ZZ}
	Fix $w\in \S^N_+$ and $T>0$. 
	Assume that the function $(s,u)\mapsto C_t(s,u)$ is continuous, for each $t\leq T$, such  that 
		\begin{align}\label{eq:boundC}
	\sup_{t\leq T} \sup_{t\leq s,u\leq T}  |C_{t}(s,u)| < \infty.
	\end{align}
	{Assume that $t\mapsto \bold{C}_t$ is strongly differentiable on $[0,T]$ with derivative \eqref{eq:diffC}.}
Then, 
	\begin{align}
	\E\left[  \exp\left( -\int_t^T Z_s^\top w Z_s ds  \right) \Mid \Fc_t \right]= \exp\left({\phi}_{t,T} + \int_{(t,T]^2} g_{t}(s)^\top {\Psi}_{t,T}(ds,du)  g_{t}(u) \right) , \quad t \leq {T},
	\end{align}
	 where $t\mapsto \Psi_{t,T}$ is given by \eqref{eq:psi_tT} and $\phi_{t,T}$ by
	\begin{align*}
	\dot \phi_{t,T} &=   -  \int_{(t,T]^2} \tr\left(  \Psi_{t,T}(ds,du) K_T(u,t) K_T(s,t)^\top \right), \quad 	\phi_{T,T}= 0. 
	\end{align*}
	In particular, $t\mapsto \Psi_{t,T}$ solves the Riccati equation with moving boundary
	\begin{align}
	\dot \psi_{t,T} &= 2 \Psi_{t,T} \star \dot C_{t} \star \Psi_{t,T} \quad\quad  \mbox{on } (t,T]^2  \quad a.e., \label{eq:Ricc00}\\
	\psi_{t,T}(t,\cdot)&=\psi_{t,T}(\cdot,t)^\top =0 \quad \,\;\;\; \,\;\,\,\,\, \mbox{on } \, [t,T] \,\, \quad a.e.\label{eq:Ricc01}
	\end{align}
\end{theorem}
  
  We note  that, since $\psi_{t,T}(s,u)=0$ whenever $s\wedge u\leq t$, equation \eqref{eq:Ricc00} is the compact form of 
  \begin{align*}
 \dot \psi_{t,T}(s,u) &=  - 2  wK_T(s,t)K_T(u,t)^\top  w  \\
  &\quad   -   2 w K_T(s,t)\int_t^T K_T(z,t)^\top  \psi_{t,T} (z,u) dz \\
  &\quad  -  2 \int_t^T \psi_{t,T}(s,z)   K_T(z,t) dz  K_T(u,t)^\top w   \\
  &\quad   -  2 \int_t^T    \psi_{t,T} (s,z)  K_T(z,t) dz \int_t^T K_T(z',t)^\top   \psi_{t,T}(z',u)dz', \quad t <s,u \leq T \; a.e.
  \end{align*}
  and the expanded form of $\phi$ is given by
    \begin{align*}
   \dot \phi_{t,T}&=    \int_t^T \tr\left( w K_T(s,t)  K_T(s,t)^\top  \right) ds-\int_t^T\int_t^T \tr\left(  \psi_{t,T}(s,u)   K_T(u,t) K_T(s,t)^\top  \right)dsdu.
    \end{align*}


\begin{remark}
	The Riccati equation \eqref{eq:Ricc00} can be compared to the \citet{bellman1957functional} and \citet{krein1955new} variation formula for Fredholm's resolvent, see also \citet{golberg1973generalization,schumitzky1968equivalence}.
\end{remark}

\section{The Volterra Wishart process and its Laplace transforms}\label{S:Wishart}
Fix $T>0$ and a filtered probability space $(\Omega,\Fc, (\Fc_t)_{t\leq T} ,\P)$ supporting a $d\times m$--matrix valued Brownian motion $W$. In this section, we consider the special case of the matrix-valued Volterra Gaussian process  
\begin{align}\label{eq:volterragaussian}
X_t &= g_0(t) + \int_0^t K(t,s)  dW_s, 
\end{align}
where  $g_0:[0,T] \to \R^{d \times m}$ is continuous and $K: [0,T] \to \R^{d \times d}$ is a  $d \times d$--measurable kernel of Volterra type, that is $K(t,s)=0$ for $s>t$. Compared to \eqref{eq:Zrep}, since the kernel $K$ is of Volterra type, the integration in  \eqref{eq:volterragaussian} goes  up to time $t$ rather than $T$.  

  Under the assumption
\begin{align}\label{eq:assumptionK}
\sup_{t \leq T} \int_0^T |K(t,s)|^2 ds< \infty \; \mbox{ and } \; \lim_{h \to 0} \int_0^T |K(u+h,s)-K(u,s)|^2 ds=0, \;\;\; u \leq T,
\end{align}  
the stochastic convolution
$$ N_t = \int_0^t K(t,s)dW_s,$$
is well defined as an It\^o integral, for each $t\in [0,T]$. Furthermore, It\^o's isometry leads to
\begin{align}\label{eq:meansquareN}
\E \left[ \left | N_t - N_s\right| ^2 \right] \leq 2  \int_s^t |K(t,u)|^2 du + 2 \int_0^T |K(t,u)-K(s,u)|^2 du   
\end{align}
which goes to $0$ as $s \to t$ showing that $N$ is mean-square continuous, and by virtue of \citet[Proposition 3.21]{PZ07}, the process $N$  admits a predictable version. Furthermore, by the Burkholder-Davis-Gundy inequality {applied on the local martingale $\left(\int_0^r K(t,s)dW_s\right)_{r\in [0,t)}$,}  it holds that 
	\begin{align}\label{eq:estimateX}
\sup_{t \leq T} \E \left[ \left| \int_0^t K(t,s)dW_s\right|^p \right] 	 \leq c_{p,T} \left(\sup_{t \leq T} \int_0^T |K(t,s)|^2 ds\right)^{p/2} < \infty, \quad   p \geq 2,
\end{align}
where $c_{p,T}$ is a positive constant  only depending on $T$ and $p$. Kernels satisfying \eqref{eq:assumptionK} are known as  Volterra kernels of continuous and  bounded type  in $L^2$ in the terminology of \citet[Definitions 9.2.1, 9.5.1 and 9.5.2]{GLS:90}.

 We now provide several kernels of interest that satisfy  \eqref{eq:assumptionK}. In particular, we stress that \eqref{eq:assumptionK}  does not exclude  a singularity of the kernel at $s=t$. 

\begin{example}
	\begin{enumerate}
\item 
For $H\in (0,1)$, the fractional Brownian motion  with  covariance function \eqref{eq:covfbm} admits a Volterra representation of the form \eqref{eq:volterragaussian} on $[0,T]$ with  the kernel
\begin{align*}
K_H(t,s)= \frac{(t-s)^{H-1/2}}{\Gamma(H+\frac 1 2)} \, {}_2 F_1\left(H-\frac 1 2; \frac 1 2-H; H+\frac 1 2; 1-\frac t s \right), \quad s \leq t,
\end{align*}
where ${}_2F_1$ is the Gauss hypergeometric integral, see  \citet{decreusefond1999stochastic}. 
		\item
		If $K$ is continuous on $[0,T]^2$, then  \eqref{eq:assumptionK} is satisfied by boundedness and the dominated convergence theorem. This is the case for instance for the Brownian Bridge $W^{T_1}$ conditioned to be equal to $W^{T_1}_0$ at a time $T_1$: for all $T<T_1$, $W^{T_1}$ admits the Volterra representation \eqref{eq:volterragaussian} on $[0,T]$ with the continuous kernel 
		$ K(t,s)  = (T_1-t)/(T_1-s)$, for all $s,t\leq T$.
		\item 
		If $K_1$ an $K_2$ satisfy  \eqref{eq:assumptionK} then so does $K_1\star K_2$ by an application of  Cauchy-Schwarz inequality.
	\item 
Any convolution kernel of the form $K(t,s)=k(t-s)\bold 1_{s\leq t}$	with $k\in L^2([0,T],\R^{d\times d})$ satisfies \eqref{eq:assumptionK}. Indeed, for any $t\leq T$,
$$   \int_0^T |K(t,s)|^2 ds = \int_0^t |k(t-s)|^2 ds = \int_0^t |k(s)|^2 ds \le \int_0^T |k(s)|^2 ds,$$
yielding the first part of \eqref{eq:assumptionK}. The second part follows from the  $L^2$-continuity of $k$, see \cite[Lemma 4.3]{brezis2010functional}.
	\end{enumerate}
\end{example}

We   denote the conditional expectation of $X$ by
\begin{align}\label{eq:meanconditional}
g_t(s)=\E\left[X_s |\Fcal_t\right ],\quad  t\leq s\leq T,
\end{align}
which is well-defined thanks to \eqref{eq:estimateX}. For each $t\geq 0$, we denote by $C_t$ the conditional covariance function of $X$ with respect to $\Fc_t$, that is 
\begin{align}\label{eq:meanvariance}
C_{t}(s,u) = \int_t^{s\wedge u}  K(s,r) K(u,r)^\top  dr, \quad t \leq s,u \leq T.
\end{align}
 $C$ satisfies the assumption of Theorem~\ref{T:char2ZZ} as shown in the next lemma. {The expression of the strong derivative of $\boldsymbol{C}_t$ is given in terms of the density $\dot{C}_t$ of the kernel $C_t$  given by \eqref{eq:Ctdensity} under the following additional assumption on the kernel:
\begin{align}\label{eq:assumptionkerneldiff}
\sup_{t\leq T} \int_0^T |K(s,t)|^2 ds < \infty.
\end{align}}
\begin{lemma}\label{L:Ccont}
Under \eqref{eq:assumptionK},  $(s,u)\mapsto C_t(s,u)$ is continuous, for all $t\leq T$ and  \eqref{eq:boundC} holds.  Furthermore, under \eqref{eq:assumptionkerneldiff},	$t\to \boldsymbol{C}_t$ is strongly differentiable on $[0,T]$ with derivative $\dot{\boldsymbol{C}}_t$  at $t\leq T$ given by the integral operator induced by the kernel $\boldsymbol{C_t}$ given by \eqref{eq:Ctdensity}, that is
\begin{align}
(\dot{\boldsymbol{C}}_t f)(s) = \int_0^T \dot{C}_t(s,u) f(u) du=\int_t^T \dot{C}_t(s,u) f(u) du , \quad f \in L^2\left( [0,T],\R^N \right). \qquad
\end{align}
\end{lemma}

\begin{proof} {First, it follows from \eqref{eq:meansquareN} that the process $X$ is mean-square continuous, which implies the continuity of $(s,u)\mapsto C_t(s,u)$.}
	Second, an application of the Cauchy-Schwarz inequality on \eqref{eq:meanvariance} yields
	$$ |C_t(s,u)|^2  \leq \left(\sup_{s'\leq T}\int_0^T |K(s',r)|^2 dr\right)^2 $$
	which proves \eqref{eq:boundC}.  {Finally, to prove the differentiability statement, we  fix $t\leq T$ and first observe that 
	\begin{align*}
	\int_0^T \int_0^T  |\dot{C}_t(s,u)|^2dsdu&=  \left(\int_0^T |K(s,t)|^2 ds\right)^2
	\end{align*}
	which is finite by virtue of \eqref{eq:assumptionkerneldiff}. Whence, the kernel $\dot{C}_t$ belongs to $L^2\left([0,T]^2,\R^{N\times N} \right)$ so that it induces a  linear bounded integral operator $\dot{\bold{C}}_t$   from $L^2\left([0,T],\R^N \right)$ into istelf. We now prove that $r\mapsto\bold{C}_r$ is differentiable at $t$ with  derivative given by $\dot{\bold{C}}_t$. For this, 
	fix $f\in L^2\left([0,T],\R^N \right)$, $s\leq T$ and $h$ such that $t+h\leq T$. Using the fact that, for all $u,s\leq T$, $t\mapsto C_t(s,u)$ is absolutely continuous with density $\dot{C}_t(s,u)$, we get that 
	\begin{align}
	({\bold{C}}_{t+h} f)(s) 
	-   ({\bold{C}}_{t} f)(s) - h  (\dot{\bold{C}}_t f)(s) = \int_{0}^T \int_t^{t+h} \left(\dot{C}_r(s,u)-\dot{C}_t(s,u)\right) dr f(u)du:=A(s).
	\end{align}
	We now bound the right hand side in $L^2\left( [0,T], \R^N\right)$. Successive  applications of the Cauchy-Schwarz inequality together with the Fubini-Tonelli theorem yield
	\begin{align}
	\| A\|^2_{L^2} &=\int_0^T \left| \int_{0}^T \int_t^{t+h} \left(\dot{C}_r(s,u)-\dot{C}_t(s,u)\right) dr f(u)du \right |^2 ds \\
	&\leq  h \|f\|^2_B   \int_t^{t+h}  \int_0^T \int_{t}^T \left | \dot{C}_r(s,u)-\dot{C}_t(s,u)\right|^2  du  dsdr  
	\end{align}
	Therefore, 
	\begin{align}
	\frac{1}{h} \| \bold{C}_{t+h}-\bold{C}_{t} -h \dot{\bold{C}}_{t} \|_{\rm{op}} \leq \int_t^{t+h}  \int_0^T \int_{0}^T  \left|\dot{C}_r(s,u)-\dot{C}_t(s,u)\right|^2  du  dsdr.  
	\end{align}
	The right hand  side goes to $0$ by virtue of \eqref{eq:assumptionkerneldiff},  which ends the proof.}
\end{proof}

\subsection{A first representation}
By construction the process $XX^\top$ is $\S^d_+$--valued and  its   Laplace  transforms can be deduced from Theorems~~\ref{T:charzz} and \ref{T:char2ZZ}. Indeed, using the vectorization operator $\vecop$, which stacks the column of a $d\times m$--matrix $A$ one underneath another in a vector of dimension $N=d m$,  see Appendix~\ref{A:matrix}, the study of the  matrix valued process $X$ reduces to that of the $\R^{dm}$-valued Gaussian process $Z=\vecop(X)$ as done in Section~\ref{S:quadratic}. 

The following theorem represents the main result of the paper. 
 \begin{theorem}\label{T:charXX}  	Let $X$ be the $d\times m$--matrix valued process defined in \eqref{eq:volterragaussian} for some Volterra kernel $K$ satisfying \eqref{eq:assumptionK} and \eqref{eq:assumptionkerneldiff}.   Fix $t \leq T$.
	For any $u \in \S^d_+$, 
	\begin{align}\label{eq:charXX}
\!\!\!	\E\left[\exp\left(- \tr\left(    u X_TX_T^\top \right ) \right) \Mid \Fc_t \right] = \frac {\exp\left(-\tr\left(   u \left(I_d + 2C_{t}(T,T)u \right)^{-1}  g_t(T)g_t(T)^\top \right)\right)}{\det\left(I_d + 2 C_{t}(T,T)u\right)^{m/2}}.\;\,\;\;
	\end{align}
For any  $w \in \S^d_+$,  the Laplace transform
	\begin{align*}
	 \mathcal L_{t,T}(w)=\E\left[ \exp \left( -\int_t^T \tr\left( w X_s    X_s^\top\right)ds \right) \Mid \Fc_t\right] ,
	\end{align*} 
	is given by 
		\begin{align}\label{eq:L_tT}
	\mathcal L_{t,T}(w) =\exp\left(\phi_{t,T}  +  \int_{(t,T]^2} \tr\left(  g_t(s)^\top \Psi_{t,T} (ds,du) g_t(u)\right)\right),
	\end{align} 
	where $(\phi,\Psi)$ are defined by 
		\begin{align}
	\dot \phi_{t,T} &=   -  m\int_{(t,T]^2} \tr\left(  \Psi_{t,T}(ds,du) K(u,t) K(s,t)^\top  \right), \quad 	\phi_{T,T}= 0,  \label{eq:phiwishart}  \\
	\Psi_{t,T}(ds,du) &= -w  \delta_{\{s=u\}} (ds,du) - \sqrt w R^w_{t,T}(s,u) \sqrt w  ds du, \quad \mbox{on} \quad [t,T],  \label{eq:PSIwishart}
	\end{align}
	where $R^w_{t,T}$ is the $d\times d$--matrix valued resolvent of $(-2\sqrt w  C_t\sqrt w )$, with $C_t$  the conditional covariance function \eqref{eq:meanvariance} and $g_t$ the conditional mean given by \eqref{eq:meanconditional}. 	In particular,   $t\mapsto \Psi_{t,T}$ solves the Riccati equation with moving boundary
	\begin{align}
	\dot \psi_{t,T} &= 2 \Psi_{t,T} \star \dot C_{t} \star \Psi_{t,T} \quad \mbox{on } (t,T]^2  \quad a.e., \label{eq:RiccW00}\\
	\psi_{t,T}(t,\cdot)&=\psi_{t,T}(\cdot,t)^\top =0 \quad \,\;\,\,\,\, \mbox{on } \, [t,T] \,\, \quad a.e.,\label{eq:RiccW01}
	\end{align} 
	where $\psi_{t,T}(s,u)=\sqrt w R^w_{t,T}(s,u) \sqrt w $.
\end{theorem}

\begin{proof} 
	Setting $Z= \vecop (X)$ and  $\bm W=\vecop(W)$, an application of the vectorization operator $\vecop$ on both sides of the $d\times m$--matrix valued equation \eqref{eq:volterragaussian} yields the $N:=dm$ dimensional vector valued Gaussian process
\begin{align}\label{eq:vecX}
Z_s =  \vecop \left( {g_0}(s)\right) + \int_0^s \mathcal K(s,u)  d\bm W_u.
\end{align}
where $\mathcal K$ is the  $\R^{N\times N}$ kernel
$$   \mathcal K: (s,u) \mapsto (I_m \otimes  K(s,u))$$
coming from the relation \eqref{eq:kronecker1}, with $\otimes$ the Kronecker product. Whence, the conditional mean and covariance functions of $Z$ are given respectively by $\vecop(g_t)$ and 
\begin{align}\label{eq:mathcalC}
{\mathcal C}_t(s,u) =  (I_m \otimes  C_t(s,u) ), \quad u,s\leq T.
\end{align}
In addition, due to \eqref{eq:kronecker2}, 
\begin{align}
\tr(wX X^\top)=Z^\top (I_m \otimes w) Z, \quad  w \in \S^N.
\end{align}  
$\bullet$ We first prove \eqref{eq:charXX}.  Fix $t  \leq T$ and $u\in \S^d_+$.  An application of Theorem \ref{T:charzz}   yields
 	\begin{align}
\E\left[\exp\left(-  Z_T^\top (I_m \otimes u) Z_T  \right) \Mid \Fc_t \right] = \frac {\exp\left(- H_t(T) \right)}{\det\left(I_N + 2 \mathcal C_{t}(T,T)(I_m \otimes u) \right)^{1/2}},
\end{align}
with 
\begin{align}
H_t(T)&=\vecop(g_t(T))^\top  (I_m \otimes u) \left(I_N+ 2\mathcal C_{t}(T,T)(I_m \otimes u)  \right)^{-1}  \vecop(g_t(T)).
\end{align}
We observe that by \eqref{eq:mathcalC} and successive applications of the product rule  \eqref{eq:kronecker3}
\begin{align*}
\left(I_N+ 2\mathcal C_{t}(T,T)(I_m \otimes u)  \right)^{-1} &= \left((I_m\otimes I_d)+ 2(I_m \otimes C_{t}(T,T))(I_m \otimes u)  \right)^{-1}\\
&= \left(I_m\otimes \left(I_d+ 2  C_{t}(T,T)u\right)  \right)^{-1}\\
&= \left(I_m\otimes \left(I_d+ 2  C_{t}(T,T)u\right)^{-1}  \right)
\end{align*}
where the last equality follows from \eqref{eq:kronecker5}.  Another application of   \eqref{eq:kronecker3} combined with  \eqref{eq:kronecker2} yields that 
\begin{align}
H_t(T)=\tr\left( u\left(I_d+ 2  C_{t}(T,T)u\right)^{-1} g_t(T) g_t(T)^\top  \right).
\end{align}
Similarly,
\begin{align}
\det\left(I_N + 2 \mathcal C_{t}(T,T)(I_m \otimes u) \right) &= \det \left( I_m \otimes \left(I_d +  2  C_{t}(T,T)u \right)\right)\\
&=\det \left( I_m \otimes \left(I_d +  2  C_{t}(T,T)u \right)\right) \\
&= \det \left(I_d +  2  C_{t}(T,T)u \right)^m
\end{align}
where we used \eqref{eq:kronecker6} for the last identity.  Combining the above proves  \eqref{eq:charXX}.\\
$\bullet$ We now prove \eqref{eq:L_tT}. Fix $t  \leq T$ and $w\in \S^d_+$.
An application of Theorem~\ref{T:char2ZZ}, justified by Lemma~\ref{L:Ccont},  yields that 
\begin{align}\label{eq:tempLcal}
	\mathcal L_{t,T}(w) =\exp\left(\phi_{t,T}  +  \int_{(t,T]^2} \vecop(g_t(s))^\top \widetilde{\Psi}_{t,T} (ds,du) \vecop(g_t(u))\right),
\end{align}
where 
\begin{align}
	\dot \phi_{t,T} &=  -  \int_{(t,T]^2} \tr\left(  \widetilde \Psi_{t,T}(ds,du) \mathcal K(u,t) \mathcal  K(s,t)^\top \right), \quad 	\phi_{T,T}= 0, \label{eq:phitempwih} \\
\widetilde \Psi_{t,T}(ds,du)&= - (I_m \otimes w)\delta_{s=u} (ds,du) -  (I_m \otimes \sqrt w)  \widetilde{\mathcal{R}}^w_{t,T}(s,u)  (I_m \otimes \sqrt w)  ds du, \nonumber
\end{align}
and  $ \widetilde{\mathcal{R}}^w_{t,T}$ is the resolvent of $2 {\mathcal C}^w_t(s,u)$. The claimed expressions now follows provided we prove that 
\begin{align}\label{eq:resolventmathcal}
\widetilde {\mathcal R}^w_{t,T}=\left( I_m \otimes R^w_{t,T} \right),
\end{align}
where $R^w_{t,T}$ is the resolvent kernel of  $2 {C}^w_t(s,u)$. Indeed, if this is the case, then, using the the product rule  \eqref{eq:kronecker3}  we get that 
\begin{align}\label{eq:temcalPSi}
\widetilde \Psi_{t,T}=\left( I_m \otimes \Psi_{t,T}\right),
\end{align}
where $\Psi_{t,T}$ is given by \eqref{eq:PSIwishart}, so that, by \eqref{eq:kronecker2}, 
$$ \vecop(g_t(s))^\top \widetilde{\Psi}_{t,T} (ds,du) \vecop(g_t(u)) = \tr\left( g_t(s)^\top  {\Psi}_{t,T} (ds,du) g_t(u) \right). $$
  Plugging \eqref{eq:temcalPSi} back in  \eqref{eq:phitempwih} and using the identity \eqref{eq:kronecker4} yields  \eqref{eq:phiwishart}. Combining the above shows that \eqref{eq:tempLcal} is equal to \eqref{eq:L_tT}.
We now prove \eqref{eq:resolventmathcal}. For this, we 
define  ${\mathcal R}^w_{t,T}=\left( I_m \otimes R^w_{t,T} \right).$ 
Then,  it follows from the resolvent equation \eqref{eq:resequation} of $R^w_{t,T}$ and the product rule  \eqref{eq:kronecker3} that  ${\mathcal R}^w_{t,T}$ solves 
\begin{align*}
 {\mathcal  R}^w_{t,T} =-2\mathcal  C_{t} -2  {\mathcal  R}^w_{t,T}\star \mathcal  C_{t}, \quad {\mathcal  R}^w_{t,T}\star  \mathcal  C_{t}  = \mathcal  C_{t}  \star {\mathcal  R}^w_{t,T}, 
\end{align*}
showing that ${\mathcal  R}^w_{t,T}$ is a resolvent of  $(-2\mathcal  C_{t})$. By uniqueness of the resolvent,  see \citet[Lemma 9.3.3]{GLS:90}, \eqref{eq:resolventmathcal} holds. \\
$\bullet$ Finally, the Riccati equations~\eqref{eq:RiccW00}--\eqref{eq:RiccW01} follow along the same lines by invoking Theorem~\ref{T:char2ZZ}.
\end{proof}

\begin{remark}\label{R:approx1}
Proposition~\ref{eq:Papprox1} can be applied to the vectorized Gaussian  process  $Z=\vecop(X)$ given by \eqref{eq:vecX} to get an approximation formula for 	$$\E\left[ \exp \left( -\int_t^T \tr\left( w X_sX_s^\top   \right)ds \right) \Mid \Fc_t\right]= \E\left[ \exp \left( - \int_t^T \tr\left( Z_s^\top (I_m \otimes w) Z_s   \right)ds \right) \Mid \Fc_t\right].$$ 
{To illustrate the convergence, we consider $d=m$, $K\equiv I_d$ and   $g_0\equiv X_0 \in \R^{d\times d}$.
	In this case, the Laplace transform of the integrated process $XX^\top$ is given in the following closed form, see \citet[Theorem 1]{gnoatto2014explicit},
	$$\E\left[ \exp \left( -\int_0^T \tr\left( w X_sX_s^\top   \right)ds \right)\right]= \exp \left( - \phi(T) - \tr\left( \psi(T) X_0 X_0^{\top}\right) \right),$$
	with 
	\begin{align*}
	\phi(T)&=\frac d 2 \tr(\log(\cosh(\sqrt{2w}T)))\\
	\psi(T)&=\frac{1}{2}\left( \cosh(\sqrt{2w} T)\right)^{-1} \left( \sqrt{2w} \sinh(\sqrt{2w} T) \right).
 	\end{align*}
 	We set  $d=2$,
 	$$X_0=\left(\begin{array}{cc}
 	0.1	& 0.3 \\
 	0.2	& 0.4\\
 	\end{array}\right), \quad w =\left(\begin{array}{cc}
 	1	& 0.5 \\
 	0.5	& 1\\
 	\end{array}\right),$$
 	and we test the approximation of Proposition~\ref{eq:Papprox1}  with the left Riemann sum and the Gauss-Legendre quadrature applied to the vectorized Gaussian  process  $Z=\vecop(X)$. The convergence is illustrated in Table~\ref{tablewishartapprox2} as the  discretization step $n$ varies. 
 	\begin{table}[h!]
 		\centering  
 		\begin{tabular}{ccc } 
 			\hline
 			\hline                 
 			Reference value  &    0.1749568	   \\ 
 			  &   \\ 
 			$n$ &   Riemann & Gauss-Legendre     \\ 
 			10 & 0.1592312  & 0.1756507\\ 
 			20 & 0.1666600 &  0.1751393\\  
 			30 & 	0.1693255 &   0.1750393\\ 
 			50 & 0.1715292 &  0.1749869\\ 
 			100 &  0.1732245 & 0.1749645 \\ 
 			200 &  0.1740860 & 0.1749588 \\ 
 			500 & 0.1746074  & 0.1749572\\ 
 			1000 & 0.1747819  & 0.1749569\\ 
 			\hline
 		\end{tabular}
 		\caption{Approximation with  $n$ ranging between $10$ and $1000$. }
 		\label{tablewishartapprox2} 
 	\end{table}
}
\end{remark}

\subsection{A second representation for certain convolution kernels}\label{S:approx2}
	
	The aim of this section is to link the Volterra Wishart distribution with conventional linear-quadratic processes \citep{chen2004quadratic,cheng2007linear}  for the special case of  convolution kernels:
	\begin{align}\label{eq:Klaplace}
	K(t,s)=k(t-s)\bm 1_{s \leq t}  \quad \mbox{ such that } \quad  k(t) = \int_{\R_+} e^{-xt} \mu(dx), \quad t > 0,
	\end{align}
	where $\mu$ is a $d\times m $--measure of locally bounded variation satisfying 
		\begin{align}\label{eq:condmu}
	\int_{\R_+} \left(1 \wedge x^{-1/2}\right) |\mu|(dx) < \infty,
	\end{align}
	and $|\mu|$ is the total variation of the measure, as defined in \citet[Definition~3.5.1]{GLS:90}.  The condition  \eqref{eq:condmu} ensures that $k$ is locally square integrable, see \citet[Lemma A.1]{AJMP19a}. This is inspired by the approach initiated in  \citet{carmona2000approximation}   and   generalized to stochastic Volterra equations in  \citet{aj2019markovian, cuchiero2019markovian, harms2019affine}.

 {Several kernels of interest satisfy \eqref{eq:Klaplace}-\eqref{eq:condmu} such as weighted sums of exponentials and the Riemann-Liouville fractional kernel 	
$K_{RL}(t)=\frac{t^{H-1/2}}{\Gamma(H+1/2)},$ for  $H \in (0,1/2)$. We refer to  \citet[Example 2.2]{AJMP19a} for more examples.}

A straightforward application of stochastic Fubini's theorem provides the representation of
	$(X_t,g_t)_{t \geq 0}$  in terms of $\mu$  and the possibly infinite system of $d\times m$-matrix-valued Ornstein-Uhlenbeck processes
	$$  Y_t(x) = \int_0^t e^{-x(t-s)} dW_s, \quad  t\geq 0, \quad x \in \R_+,$$
	see for instance \citet[Theorem~2.3]{AJMP19a}.
	\begin{lemma}
		Assume that $K$ is of the form \eqref{eq:Klaplace} with $\mu$ satisfying \eqref{eq:condmu}, then 
		\begin{align*}
		X_t &= g_0(t)+ \int_{\R_+} \mu(dx) Y_t(x),  &&t\leq T,\\
		g_t(s) &= g_0(s) +  \int_{\R_+} e^{-x(s-t)}\mu(dx)Y_t(x),  &&t\leq s \leq T.
		\end{align*}  
	\end{lemma}
	
	Combined with \eqref{eq:L_tT}, we get an exponentially quadratic representations of the characteristic function of $XX^\top$ in terms of the process $Y$.
	\begin{theorem}
		Assume that $K$ is of the form \eqref{eq:Klaplace} with $\mu$ satisfying \eqref{eq:condmu} and fix $w\in \S^d_+$.  Then,
		\begin{align}
		\mathcal L_{t,T}(w) = \exp\Bigg( \Theta_{t,T} &+ 2 \tr\left( \int_{\R_+}\Lambda_{t,T}(x)^\top \mu(dx) Y_t(x) \right) \\
		&\quad \quad  +\tr\left( \int_{\R_+^2}  Y_t(x)^\top \mu(dx)^\top   \Gamma_{t,T}(x,y) \mu(dy) Y_t(y) \right) \Bigg) \label{eq:charXXaffine},
		\end{align}
		where  $t\mapsto (\Theta_{t,T},\Lambda_{t,T},\Gamma_{t,T})$ are given by
		\begin{align}
		\Theta_{t,T}&= \int_{(t,T]^2}\tr \left(  g_0(s)^\top \Psi_{t,T}(ds,du)g_0(s)  \right) + \phi_{t,T}, \label{eq:Ricmu1}\\
		\Lambda_{t,T}(x)&= \int_{(t,T]^2}  e^{-x(s-t)} \Psi_{t,T}(ds,du)g_0(u), \label{eq:Ricmu2}\\
		\Gamma_{t,T}(x,y)&= \int_{(t,T]^2} e^{-x(s-t)} \Psi_{t,T}(ds,du)  e^{-y(u-t)} , \label{eq:Ricmu3}
		\end{align}
		with $(\phi,\Psi)$ as in \eqref{eq:phiwishart}-\eqref{eq:PSIwishart}.
	\end{theorem}
	
A direct differentiation of $(\Theta,\Lambda,\Gamma)$ combined with the Riccati equations~\eqref{eq:RiccW00}--\eqref{eq:RiccW01} for $(\phi,\Psi)$ yield a  system of Riccati equation for $(\Theta,\Lambda,\Gamma)$. 
	\begin{proposition}
		The functions $t\mapsto (\Theta_{t,T},\Lambda_{t,T},\Gamma_{t,T})$ given by \eqref{eq:Ricmu1}, \eqref{eq:Ricmu2} and \eqref{eq:Ricmu3} solve the system of  backward Riccati  equations
		\begin{align}
		\dot \Theta_{t,T} &= -\mathcal R_0(t,\Lambda_{t,T}, \Gamma_{t,T}),  && \Theta_{T,T}=0 , \label{eq:Theta}\\
			\dot \Lambda_{t,T}(x) &= x  \Lambda_{t,T}(x) - \mathcal R_1(t,\Lambda_{t,T}, \Gamma_{t,T})(x),  && \Lambda_{T,T}(x)=0,  \label{eq:Lambda}\\
		\dot \Gamma_{t,T}(x,y) &= (x+y)  \Gamma_{t,T}(x,y) - \mathcal R_2( \Gamma_{t,T})(x,y),  && \Gamma_{T,T}(x,y)=0,  \label{eq:Gamma} 
		\end{align}	
		where
		\begin{align*}
		\mathcal R_0(t,\Lambda,\Gamma) &= -\tr\left( g_0(t)^\top w g_0(t)\right) +  m \tr \left( \int_{\R_+^2} \Gamma(x,y)\mu(dy) \mu(dx )^\top  \right) \\
		&\quad \; + 2\tr\left(  \left(\int_{\R_+} \Lambda(x)^\top\mu(dx) \right)  \left(\int_{\R_+} \Lambda(y)^\top \mu(dy) \right)^\top \right), \\
		\mathcal R_1(t,\Lambda,\Gamma)(x)&= -wg_0(t)+ 2 \left(\int_{\R_+}  \Gamma(x,x') \mu(dx')\right) \left(\int_{\R_+} \Lambda(y)^\top  \mu(dy)\right)^\top,\\
		\mathcal R_2(\Gamma)(x,y)&= -w + 2 \left(\int_{\R_+} \Gamma(x,x') \mu(dx')\right) \left(\int_{\R_+} \Gamma(y,y') \mu(dy')\right)^\top.
		\end{align*}
	\end{proposition}

	Similar Riccati equations to that of $\Gamma$ have appeared in the literature when dealing with convolution kernels of the form \eqref{eq:Klaplace} in the presence of a quadratic structure, see  \citet[Theorem 3.7]{AJMP19a}, \citet[Theorem 1]{alfonsi2013capacitary},  \citet[Lemma~5.4]{harms2019affine}, \citet[Corollary~6.1]{cuchiero2019markovian}. A general existence and uniqueness result for more general equations  has been recently obtained in \citet{AJMP19b}.
	
	\begin{remark}
		The expression~\eqref{eq:charXXaffine} can be re-written in the following compact form 
			\begin{align*}
	\mathcal L_{t,T}(w)= 
		\exp\left( \Theta_{t,T}+ 2\langle  \Lambda_{t,T}, Y_t \rangle_{\mu}	+ \langle  Y_t , \bold{\Gamma}_{t,T} Y_t\rangle_{\mu}   \right) .
		\end{align*}
		where $\bold{\Gamma}_{t,T}$ is the integral operator  acting on $L^1(\mu,\R^{d\times m})$ induced by the kernel ${\Gamma}_{t,T}$:
		\begin{align*}
	( \bold{\Gamma}_{t,T} f) (x) = \int_{\R_+} {\Gamma}_{t,T} (x,y) \mu(dy) f(y),  \quad f \in  L^1(\mu, \R^{d\times m})
		\end{align*}
		and $\langle \cdot,\cdot\rangle_{\mu}$ is the dual pairing
		\begin{align*}
		\langle  f, g\rangle_{\mu} = \tr \left( \int_{\R_+} f(x)^\top \mu(dx)^\top g(x)   \right),\quad (f,g) \in L^{1}(\mu, \R^{d\times m})\times L^{\infty}(\mu^\top, \R^{d\times m}).
		\end{align*}
	\end{remark}

	We end this subsection with two examples establishing the connection with conventional quadratic models. 
	\begin{example} Fix $\Sigma\in \R^{d \times d}$. For the constant case $k \equiv \Sigma$ we have $\mu({dx}) = \Sigma \delta_0(dx)$, $\supp \mu =\{0\}$ and $Y_t(0)=W_t \in \R^{d\times m}$. For $g_0(t)\equiv0$, $\Lambda\equiv 0$ and \eqref{eq:Theta}  and  \eqref{eq:Gamma}  read
		\begin{align*}
	\dot \Theta_{t,T} &= - m \tr\left(    \Gamma_{t,T}(0,0) \Sigma \Sigma^\top\right),  && \Theta_{T,T}=0 , \\
	\dot 	\Gamma_{t,T}(0,0) &=  w -2 \Gamma_{t,T}(0,0)  \Sigma \Sigma^\top \Gamma_{t,T}(0,0) ,  && \Gamma_{T,T}(0,0)=0. 
		\end{align*}
		These are precisely the conventional backward matrix Riccati equations encountered  for conventional Wishart processes, see \citet[Equation (5.15)]{A15}.  In this case, we recover the well-known Markovian expression for the conditional Laplace transform \eqref{eq:charXXaffine}:
		\begin{align*}
		\E\left[ \exp \left( \int_t^T \tr\left( -w W_sW_s^\top   \right)ds \right) \Mid \Fc_t\right]  = \exp\left( \Theta_{t,T} +\tr\left(  \Gamma_{t,T}(0,0) W_tW_t^\top \right)\right).
		\end{align*}	
	\end{example}

	\begin{example}\label{E:weightexp} Fix $n\geq 1$, $x^n_i \in \R_+$  and $c^n_i \in \R^{d\times d}$, $i=1,\ldots,n$. Consider the kernel
		\begin{align}\label{eq:kn}
		k^n(t)=\sum_{i=1}^n c^n_i e^{-x^n_i t }, \quad t\geq 0,
		\end{align}
		which corresponds to the measure $\mu^n(dx)= \sum_{i=1}^n c^n_i \delta_{x^n_i}(dx) $. 
		The system of Riccati equations   \eqref{eq:Theta}, \eqref{eq:Lambda}  \eqref{eq:Gamma}  is reduced to a system of finite dimensional matrix Riccati equations for
		with values in $\R\times \R^{nd\times m}\times \R^{nd\times nd}$ given by:
			\begin{align}
		\dot \Theta_{t,T}^n &=  \tr\left(g_0(t)^\top w g_0(t)\right) -  m \tr \left(  \Gamma^n_{t,T} C^n  \right)  - 2\tr\left(  \Lambda^{n \top}_{t,T} C^n \Lambda^{n}_{t,T} \right)  ,  &&\!\!\!\!\!\!\!\!\!\Theta^n_{T,T}=0, \label{eq:Thetan} \\
		\dot \Lambda^n_{t,T} &= D^n_t + B^n \Lambda^n_{t,T}  - 2 \Gamma^n_{t,T} C^n \Lambda^n_{t,T},  &&\!\!\!\!\!\!\!\!\!\Lambda^n_{T,T}=0,   \label{eq:Lambdan}\\
		\dot \Gamma^n_{t,T} &= A^n + B^n \Gamma^n_{t,T} + \Gamma^n_{t,T}  B^{n\top} - 2 \Gamma^n_{t,T} C^n \Gamma^n_{t,T},  &&\!\!\!\!\!\!\!\!\!\Gamma^n_{T,T}=0, \label{eq:Gamman}
		\end{align}
		where  for all $r=1,\ldots,m$ $i,j,=1,\ldots,n$ and $k,l=1,\ldots,d$, $p=(i-1)d + k ,q=(j-1)d + l,$
	\begin{align}
	(D^n_t)^{pr}&=(wg_0(t))^{kr},
	&&(\Lambda^n_{t,T})^{p r}= \Lambda_{t,T}(x_i)^{kr},\\
			(C^n)^{pq}&= (c_i^n c_j^{n\top})^{kl},
	&&(\Gamma^n_{t,T})^{pq}=\Gamma_{t,T}(x_i^n,x_j^n)^{kl}, 
	\end{align}
and $A^n$ and $B^n$ are the $nd\times nd$ defined by
	\begin{align}
	A^n=(\mathbbm 1_n \otimes w),  \quad B^n=(\diag(x_1^n,\ldots,x^n_n)   \otimes I_d) 
	\end{align}
	with $\mathbbm 1_n$ the $n\times n$ matrix with all components equal to $1$.   The Riccati equation \eqref{eq:Gamman} can be linearized by doubling the dimension and its solution  is given  explicitly by 
	\begin{align*}
	\Gamma^n_{t,T}= G_2(T-t) G_4(T-t)^{-1}, \quad t\leq T,
	\end{align*}
	where 
	\begin{align}
\left(	\begin{matrix}
	G_1(t) & G_2(t) \\
	 	G_3(t) & G_4(t) \\
	\end{matrix} \right) = \exp\left( t \left(	\begin{matrix}
-	B^n & -A^n \\
-2C^n & B^n \\
	\end{matrix} \right)  \right), \quad t\leq T,
	\end{align}
	see \citet{levin1959matrix}.  Furthermore, we recover the well-known Markovian expression for the conditional Laplace transform \eqref{eq:charXXaffine}:
	\begin{align}\label{eq:Ln}
	\mathcal L_{t,T}^n(w) = \exp\left( \Theta^n_{t,T} + 2\tr\left(  \Lambda^{n\top}_{t,T} \widetilde Y_t^n\right)  +\tr\left(  \Gamma^n_{t,T} \widetilde Y_t^n \widetilde Y_t^{n\top}\right)\right),
	\end{align}	
	where 
	\begin{align}\label{eq:Yn}
	(\widetilde Y^n_t)^{pr}=(c_i^n Y_t(x_i^n))^{kr},
	\end{align}
	$p=(i-1)d + k$, for each $i=1,\ldots,n$, $k=1,\ldots,d$, $r=1,\ldots,m$.
\end{example}

The previous example shows that Volterra Wishart processes can be seen as a superposition of possibly infinitely many conventional linear-quadratic processes in the sense of \citet{chen2004quadratic,cheng2007linear}. This idea is used to build another approximation procedure in the next subsection. 
	
	\subsection{Another approximation procedure}\label{S:approx2}

An application of the Burkholder-Davis-Gundy inequality  yields the following stability result for the sequence 
		\begin{align*}
X^n_t &= g^n_0(t) + \int_0^t k^n(t-s)dW_s, \quad n \geq 1,
\end{align*} 
where $g_0^n: [0,T]\to \R^{d\times m}$ and $k^n \in L^2([0,T],\R^{d\times d})$, for each $n\geq 1$.
	\begin{lemma}\label{L:stability}
		Fix $k \in L^2([0,T],\R^{d\times d})$ and $g_0:[0,T]\to \R^{d\times m}$ measurable  and bounded. If 
		\begin{align}\label{eq:assumptionconvergence}
		\int_0^T |k^n(s) - k(s)|^2ds  \to 0 \quad \mbox{and} \quad 	\sup_{t \leq T} |g_0^n(t)-g_0(t)| \to 0, \quad \mbox{as \, $n \to \infty$},
		\end{align}
		then, 
		\begin{align}
		\sup_{t \leq T} \E\left[ \left|X^n_t- X_t\right|^p \right] \to 0, \quad \mbox{as \, $n \to \infty$}, \quad p \geq 2.
		\end{align}	
	\end{lemma}
	
	Combined with Example~\ref{E:weightexp}, we obtain another approximation scheme  for the Laplace transform based on finite-dimensional matrix Riccati equations (compare with Remark \ref{R:approx1}).
	\begin{proposition}
		Fix $w\in \S^d_+$ and $t\leq T$. For each $n$, let $k^n$ be as in \eqref{eq:kn} for some  $x^n_i \in \R_+$  and $c^n_i \in \R^{d\times d}$. 
		Assume that \eqref{eq:assumptionconvergence} holds. Then, 
		\begin{align}
		\mathcal L_{t,T}(w)= \lim_{n \to \infty}\exp\left( \Theta^n_{t,T} + 2 \tr\left(  \Lambda^{n\top}_{t,T} \widetilde Y_t^n\right)  +\tr\left(  \Gamma^n_{t,T} \widetilde Y_t^n \widetilde Y_t^{n\top}\right)\right)
		\end{align}
		where $(\Theta^n,\Lambda^n,\Gamma^n)$ solve \eqref{eq:Thetan}, \eqref{eq:Lambdan} and \eqref{eq:Gamman} and $\widetilde Y^n$ is given by \eqref{eq:Yn}.
	\end{proposition}

\begin{proof}
	Fix $t\leq s\leq T$. Writing $X^{n\top}_s w X^n_s -X_s^\top w X_s = (X^{n}_s +X_s)^\top w (X^n_s -X_s)$, we get by the Cauchy-Schwarz inequality that 
	$$  \E\left[ \left| \int_t^T (X^{n\top}_s w X^n_s - X_s^\top w X_s) ds \right|^2 \right] \leq  c  \sup_{s\leq T} \left(       \E\left[ |X_s|^2  \right] +\E\left[ |X^n_s|^2  \right]    \right)  \sup_{s\leq T}\E\left[ |X^n_s-X_s|^2  \right],  $$
	for some constant $c$ independent of $n$. It follows from Lemma~\ref{L:stability} that 
	$(\sup_{s\leq T}\E\left[ |X^n_s|^2  \right])_{n\geq 1 }$ is uniformly bounded in $n$,  so that the right hand side converges to $0$ as $n\to \infty$. Whence, $\int_t^T X^{n\top}_s w X^n_s ds\to \int_t^T X^{\top}_s w X_s ds$ a.s.~along a subsequence and the claimed convergence follows from the  dominated convergence theorem combined with \eqref{eq:Ln}. 
\end{proof}
	
	For  $d=m=1$ and $k$ of the form \eqref{eq:Klaplace} for some measure $\mu$, for  suitable partitions $(\eta^n_i)_{0\leq i \leq n}$ of $\R_+$, the choice 
	\begin{align*}
	c^n_i =  \int_{\eta_{i-1}^n}^{\eta_{i}^n} \mu(dx) \quad \mbox{and} \quad x^n_i = \frac 1 {c^n_i}\int_{\eta_{i-1}^n}^{\eta_{i}^n} x \mu(dx), \quad i=1,\ldots, n,
	\end{align*}
	ensures the  $L^2$-convergence of the kernels $k^n$ in \eqref{eq:kn}, we refer to \citet{aj2018lifting,aj2018multi} for such constructions,    see also \citet{harms2019strong} for other choices of quadratures and for a detailed study of strong convergence rates.

\section{Applications}\label{S:applications}

\subsection{Bond pricing in quadratic Volterra short rate models with default risk}
We consider a quadratic short rate model of the form
$$r_t = \tr \left( X_t^\top Q X_t \right) + \xi(t), \quad t\leq T,$$
 where $X$ is the $d\times m$ Volterra process as in \eqref{eq:volterragaussian},  $Q\in \S^d_+$ and 
$\xi:[0,T]\to \R$ is an input curve used to match today's yield curve and/or  control the negativity level of the short rate. The model replicates the asymmetrical distribution of
interest rates, allows for rich auto-correlation structures, and the possibility to account for long range dependence,  see for instance \citet{benth2018non,corcuera2013short}.

An application of Theorem~\ref{T:charXX} yields the price $P(\cdot,T)$   of a zero-coupon bond  with maturity $T$:
\begin{align*}
P(t,T)&=\E\left[ \exp\left( -\int_t^T r_sds   \right)\Mid \Fc_t\right] = \exp\left( -\int_t^T \xi(s) ds \right) \mathcal L_{t,T}(Q), \quad t\leq T,
\end{align*}
where $\mathcal L$  is given by \eqref{eq:L_tT}. {In this case, the zero-coupon yield with time to maturity $\tau=T-t$ is quadratic in $g$ and  reads
\begin{align} 
 y_t(\tau)&= - \frac{1}{T-t} \log P(t,T) \\
&= \frac{1}{\tau} \int_t^{t+\tau} \xi(s) ds +    \frac{1}{\tau} \langle g_t,  \bold{\Psi}_{t,{t+\tau}}   g_t  \rangle_{L^2_t} + \frac {m}{2\tau} \log  \det\left( \id + 2\sqrt Q \bold{C}_{t,t+\tau} \sqrt w  \right), 
\end{align}
with $\bold{\Psi}_{t,t+\tau}=\sqrt Q \left( \id + 2\sqrt Q \bold{C}_{t,t+\tau} \sqrt Q  \right)^{-1}\sqrt Q$ and $\langle f,g \rangle_{L_t}=\int_t^{t+\tau} \tr(f(s)^\top g(s))ds$. The role of the input curve $\xi$ becomes apparent: it allows to perfectly match any given yield curve and/or possibly push the yields into negative territory (observe that $\bold{\Psi}_{t,t+\tau}$ is a non-negative operator). Furthermore, various shapes of the deformation  of the yield curve can be replicated. For instance,  the left graph of Figure~\ref{fig:sensi} shows that in the one dimensional setting, when $X=W^H$ with $W^H$ a fractional Brownian motion with Hurst index $H\in (0,1)$, the variation of the Hurst index $H$   can produce inverse, hump-shaped and normal yield curves. The combination of two independent fractional Brownian motion with different Hurst indices lead to richer deformations as displayed on the right graph of Figure~\ref{fig:sensi}.

\begin{center}
	\includegraphics[scale=0.62]{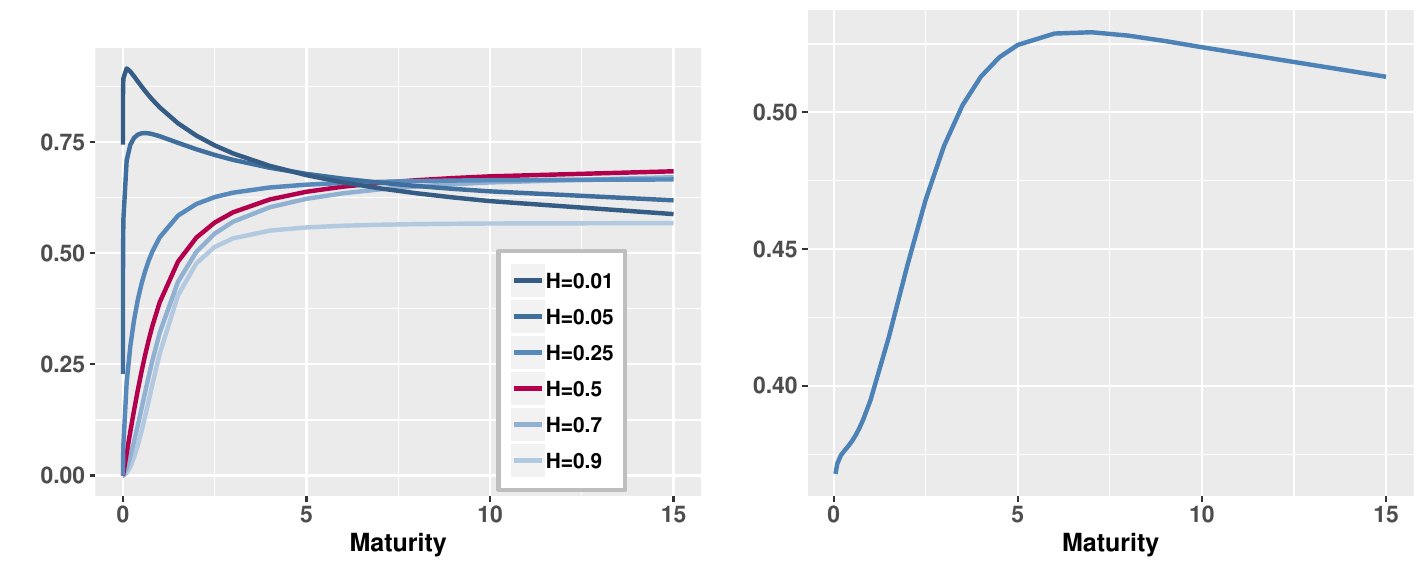}
	\rule{35em}{0.5pt}
	\captionof{figure}{{Left: Sensitivity of the yield curve  $T \mapsto y_0(T)$ with respect to the Hurst index for $d=m=1$, $X=W^{H} $ and $\xi\equiv 0$. Right: yield curve  $T \mapsto y_0(T)$  for $d=2$, $m=1$, $X=(W^{H_1},W^{H_2})^\top$, $(H_1,H_2)=(0.05,0.9)$,  $(Q_{11}, Q_{12},Q_{22})=(0.4 , 0.05, 0.1)$  and $\xi\equiv 0$}}
	\label{fig:sensi}
\end{center}

From the dynamical perspective, for two maturities $\tau_1,\tau_2$, the instantaneous covariance of the variation of the  yields  over time is given by 
\begin{align}
d\langle y_{\cdot}(\tau_1), y_{\cdot}(\tau_2)  \rangle_t &= \frac{4}{\tau_1\tau_2} \tr\left( (\boldsymbol{K}^*\boldsymbol{\Psi}_{t,t+\tau_1})(\bm 1_t g_t)(t) \left( 	 (\boldsymbol{K}^*\boldsymbol{\Psi}_{t,t+\tau_2})(\bm 1_t g_t)(t)\right)^\top \right) dt,
\end{align}
which is stochastic, non-trivial and allow  sign changing across time. 
For instance, for the standard case $K(t,s)= e^{-B(t-s)} \eta$ and $g_0(t)\equiv X_0$, with $B,\eta, X_0 \in \R^{d\times d}$, we have that $g_t(s)=X_t$ for all $s \geq t $, one recovers the expression  of the instantaneous covariance in a Wishart short rate model (see \citet{buraschi2010correlation}):
\begin{align}
d\langle y_{\cdot}(\tau_1), y_{\cdot}(\tau_2)  \rangle_t &= \frac{4}{\tau_1\tau_2} \tr\left(  \bar \Psi_{t,\tau_1}X_tX_t^\top \bar \Psi_{t,\tau_2}^\top  \right) dt
\end{align}
with $ \bar \Psi_{t,\tau_i}= \int_{(t,t+\tau_i]^2}   e^{-B(s-t)} \eta \Psi_{t,t+\tau_i}(ds, du)$.

Compared to the standard case, more general kernels allow to capture both  ``time series" and ``cross section" features of interest rates even with one single factor, see for instance \citet{backus1993long,dai2003term,ritchken2000interest}:
\begin{itemize}
	\item 
	 Figure \ref{fig:shortsimul} highlights the  auto-correlation structure of short rates: for $H=0.9$ the rates are highly persistent as observed in practice;
	 \item
	  Figure \ref{fig:varianceyields}  shows the term structure of the variance of  the yields: the case $H=0.9$ allows to reproduce a humped term structure  decaying at a slower rate than exponential (i.e. for the standard case $H=0.5$) in agreement with the empirical observations. 
\end{itemize} 
The impact is amplified in a multifactor setting where all the factors share the same Hurst index, as shown by the principal component analysis on Figure~\ref{fig:multifactorpca}.  It might be interesting to consider a mixture of several factors with different Hurst indices to better capture the behavior of yields across several maturities.

\begin{center}
	\includegraphics[scale=0.62]{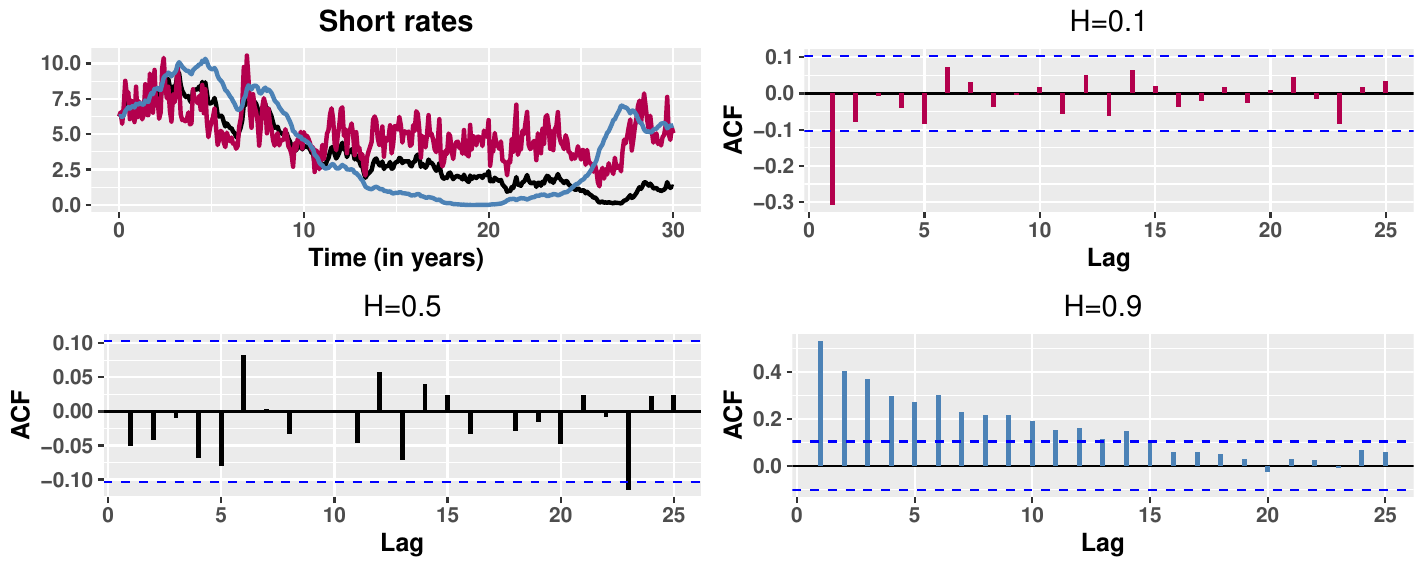}
	\rule{35em}{0.5pt}
	\captionof{figure}{{Simulation of monthly short rates (top left)  with $X_t=2.5+\frac{1}{\Gamma(H+1/2)}\int_0^t (t-s)^{H-1/2}dW_s$ and varying $H$ index: $H=0.1$ (red), $H=0.5$ (black) and $H=0.9$ (blue) with the corresponding autocorrelation plots.}}
	\label{fig:shortsimul}
\end{center}

\begin{center}
	\includegraphics[scale=0.62]{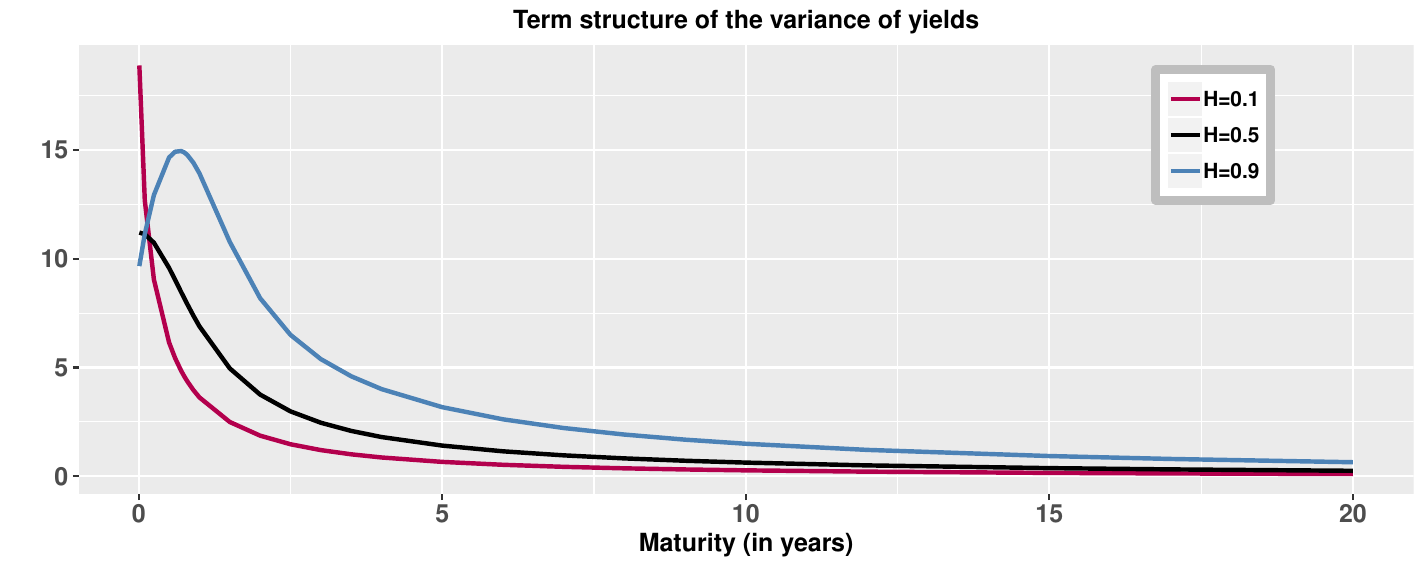}
	\rule{35em}{0.5pt}
	\captionof{figure}{{Impact of the Hurst index on the term structure of the variance of yields with $X_t=2.5+\frac{1}{\Gamma(H+1/2)}\int_0^t (t-s)^{H-1/2}dW_s$.}}
	\label{fig:varianceyields}
\end{center}

\begin{center}
	\includegraphics[scale=0.62]{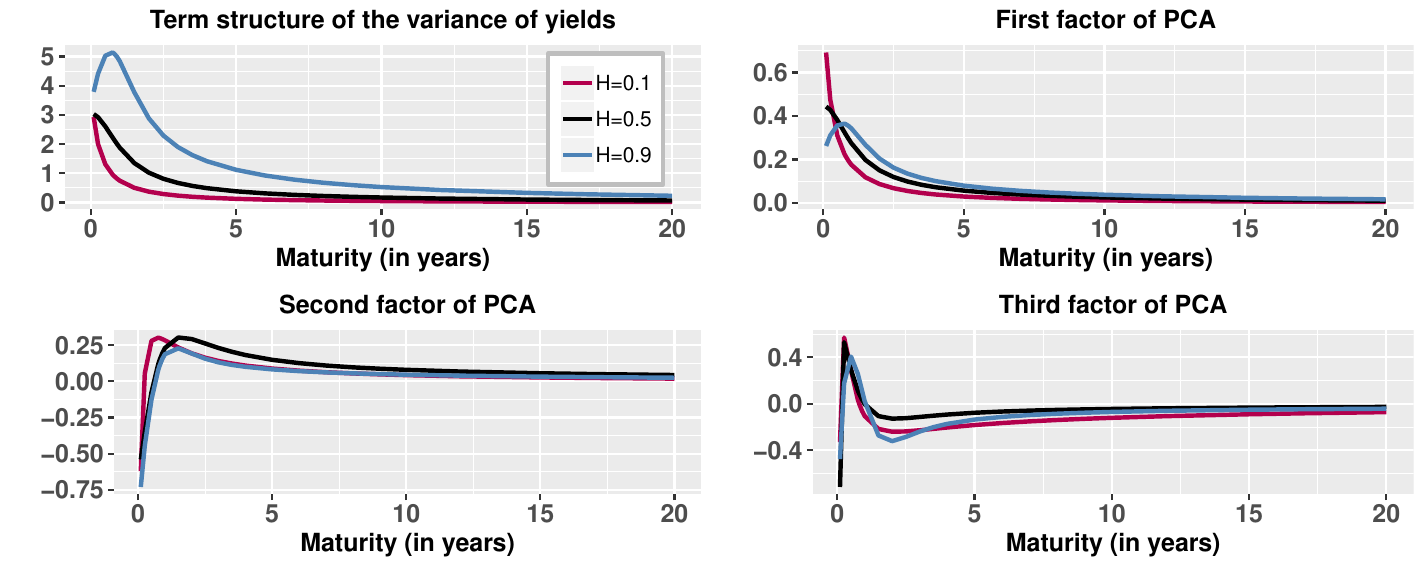}
	\rule{35em}{0.5pt}
	\captionof{figure}{{Impact of the Hurst index on the principal component analysis of the covariance of the yields with $d=3$, $m=1$, $Q_{ii}=1$ and $Q_{12}=Q_{23}=0.5$ and  $X^i_t=0.33+\frac{1}{\Gamma(H+1/2)}\int_0^t (t-s)^{H-1/2}dW^i_s$, $i=1,\ldots,3$. }}
	\label{fig:multifactorpca}
\end{center}

}

One can also add multiple spreads by considering stochastic   processes of the form  
$$\lambda_t =   \tr \left( X_t^\top \widetilde Q X_t \right) + \widetilde \xi(t), \quad t \leq T, $$
for some  $\widetilde Q\in \S^d_+$ and $\widetilde \xi:[0,T]\to \R_+$ bounded function. By definition the spread is nonnegative,  correlated to  the short rate with a possible long range dependence or roughness. The introduction of $\lambda$ can serve in two ways. Either in a multiple curve modeling framework, to add a  risky curve on top of the non-risk one with instantaneous rate $r+\lambda$ or to model default time. In the latter case, $\lambda$ would correspond to the instantaneous intensity of  a Poisson process $N$ such that the default time $\tau$ is defined as the first jump time of $N$. In both cases, we denote  by  $\widetilde P(\cdot,T)$ the price of the risky curve or the price of a  defaultable bond  paying $\bold{1}_{\tau\leq T}$ at maturity $T$. Then, on $\{t<\tau\}$, the price is given by 
\begin{align*}
\widetilde P(t,T)&=\E\left[ \exp\left( -\int_t^T (r_s +\lambda_s) ds   \right)\Mid \Fc_t\right] \\
&= \exp\left( -\int_t^T (\xi(s) +\widetilde \xi(s)) ds \right) \mathcal L_{t,T}(Q+\widetilde Q),
\end{align*}
for all $t\leq T$, we refer to \citet{lando1998cox} for more details on the derivation of the defaultable bond price.


\subsection{Pricing options on volatility/variance for basket products in Volterra Wishart covariance models}

We consider $d\geq 1$ risky assets $S=(S^1,\ldots, S^d)$ such that the instantaneous realized covariance is given by  
\begin{align}\label{eq:instantenuouscov}
d\langle  \log  S \rangle_t = X_t X_t^\top dt
\end{align}
where $X$ is the $d\times m$ process as in \eqref{eq:volterragaussian}. The following specifications for the dynamics of $S$ fall into this framework.

\begin{example}
	\begin{enumerate}
\item 
The Volterra Wishart covariance model for $d=m$:
\begin{align*}
dS_t &= \diag(S_t)  X_t dB_t , \quad S_0 \in \R^d_+\\
X_t &= g_0(t) + \int_0^t K(t,s)dW_s,
\end{align*}
with $g_0:[0,T]\to \R^{d \times d}$, a suitable measurable kernel  $K: [0,T]^2 \to \R^{d \times d}$,  
a $d\times d$ Brownian motion $W$ and 
\begin{align}\label{eq:correlatedBM}
B^j = \tr\left( W_j \rho_j^\top\right) + \sqrt{1-\tr\left(\rho_j\rho_j^\top\right)}\, W^{\perp,j},\quad j=1,\ldots,d,
\end{align}
for some  $\rho_j\in \R^{d \times m}$ such that $\tr\left(\rho_j\rho_j^\top\right)\leq 1$, for $j=1,\ldots,n$,
where  $W^\perp$ is a  $d$--dimensional Brownian motion  independent of $W$. 
		\item
The Volterra	Stein-Stein model when $d=m=1$:
	\begin{align*}
	dS_t &= S_t  X_t dB_t , \quad S_0>0,\\
	X_t &= g_0(t) + \int_0^t K(t,s)dW_s, \quad d\langle B,W\rangle_t = \rho dt,
	\end{align*}
	for some $\rho \in [-1,1]$.
	\end{enumerate}
\end{example}

The approach  of \citet{carr2008robust}, based on \citet{schurger2002laplace}, can be adapted to price various volatility and variance options on basket products. Indeed, consider a basket product of the form 
$$ P^{\alpha}_t =  \sum_{j=1}^d \alpha_j \log S^j_t = \alpha^\top \log S_t,  \quad t\leq T, $$
for some $\alpha=(\alpha_1,\ldots,\alpha_d)^\top \in \R^d$. It follows from \eqref{eq:instantenuouscov} that 
the integrated realized variance $\Sigma^{\alpha}$ of $P^{\alpha}$ is given by 
\begin{align*}
  \Sigma^{\alpha}_t=\int_0^t \alpha^\top X_s X_s^\top \alpha ds  =  \int_0^t \tr\left(\alpha \alpha^\top  X_s X_s^\top \right)ds,\quad t\leq T.
\end{align*}
Fix $q \in (0,1]$ and consider the $q$-th power variance swap whose payoff at maturity $T$ is given by
\begin{align*}
 (\Sigma^{\alpha}_T)^q - F=  \left(\int_0^T \tr\left(\alpha \alpha^\top  X_s X_s^\top \right)ds\right)^q - F,
\end{align*}
for some strike $F\geq 0$.  In particular, for $q=1/2$ one recovers a volatility swap and for $q=1$ a variance swap.  The value of the contract being null at $t=0$, the fair strike $F^*_q$ reads
$$F_q^* = \E\left[  \left(\int_0^T \tr\left(\alpha \alpha^\top  X_s X_s^\top \right)ds \right)^q   \right].$$

The following proposition establishes the  expression of the fair strike in terms of the Laplace transform provided by Theorem~\ref{T:charXX}. 
\begin{proposition}
	Assume that $0<q<1$, then the fair strike of the $q$-th power variance swap is given by 
	$$ F_q^* =  \frac q {\Gamma(1-q)} \int_0^{\infty} \frac{1-\mathcal L_{0,T}\left(z\alpha\alpha^\top\right)}{z^{q+1}} dz, $$
	where $\mathcal L$ is given by \eqref{eq:L_tT}. 
	If $q=1$, the fair strike for the variance swap reads
	$$ F_1^* =  \int_0^T \tr\left(\alpha \alpha^\top g_0(s)g_0(s)^\top ds \right)ds +  \int_0^T \int_0^s  \tr\left(\alpha \alpha^\top K(s,u)K(s,u)^\top \right) du ds. $$
\end{proposition}

\begin{proof}
	For $0<q<1$, applying the identity 
	\begin{align} 
	v^{q}=  \frac q {\Gamma(1-q)} \int_0^{\infty} \frac{1-e^{-zv}}{z^{q+1}} dz, \quad 0<q<1, \quad v\geq 0,
	\end{align}
 see \citet{schurger2002laplace}, to  $v=\int_0^T \tr\left(\alpha \alpha^\top  X_s X_s^\top \right)ds $, taking expectation and invoking   Tonelli's theorem together with Theorem~\ref{T:charXX}  yield the claimed identity. For $q=1$, one could proceed by differentiating the Laplace transform or more simply by using the dynamics of $XX^\top$ as in Remark~\ref{R:dynamic}.
\end{proof}

Similarly, one can obtain the following formulas for negative powers
\begin{align*}
 \E\left[  \left(\int_0^T \tr\left(\alpha \alpha^\top  X_s X_s^\top \right  )ds  +\epsilon\right)^{-q}   \right] =  \frac 1 {\Gamma(1+q)} \int_0^{\infty}  \mathcal L_{0,T}(z^{1/q}\alpha \alpha^\top){e^{-z^{1/q}\epsilon}} dz, \quad \epsilon,q>0,
\end{align*}
using the integral representation, taken from \citet{schurger2002laplace},
\begin{align} 
v^{-q}=  \frac 1 {q\Gamma(1+q)} \int_0^{\infty} {e^{-z^{1/q}v}} dz, \quad q,v>0.
\end{align}

Again, the approximation formulas of  Remark~\ref{R:approx1} and  Section~\ref{S:approx2} can be applied to compute $\mathcal L_{0,T}$.

\appendix

\section{Wishart distribution}\label{A:wishart}

\begin{proposition}\label{P:Wishartchar}
	Let $\xi$ be an $\R^N$ Gaussian vector with mean vector $\mu\in \R^N$ and covariance matrix $ \Sigma \in \S^{N}_+$, then $\xi\xi^\top$ follows a non-central Wishart distributions with shape parameter $1/2$, scale parameter $2\Sigma$ and non-centrality parameter $\mu \mu^\top$, written as 
	$$\xi \xi^\top \sim \mbox{WIS}_N \left(\frac 1 2 , \mu \mu^\top , 2 \Sigma\right).$$
	Furthermore, 
	\begin{align*}
	\E\left[\exp\left(- \tr\left(  u \xi \xi^\top \right ) \right)\right] = \frac {\exp\left(-\tr\left( u(I_N + 2\Sigma u)^{-1} \mu \mu^\top \right)\right)}{\det\left(I_N + 2 \Sigma u\right)^{1/2}}, \quad u \in \mathbb S^N_+.
	\end{align*}
\end{proposition}

\section{Matrix tools}\label{A:matrix}
We recall some definitions and properties of matrix tools  used in the proofs throughout the article. For a complete review and proofs we refer to \citet{magnus2019matrix}.

\begin{definition}
	The  vectorization operator   $\vecop$ is defined from $\R^{d \times m }$ to $\R^{dm}$  by stacking the columns  of a $d\times m$-matrix $A$  one underneath another in a $dm$--dimensional vector $\vecop(A)$, i.e. 
	$$ \vecop(A)_p = A_{ij}, \quad p = (j-1)d + i,$$
	for all $i=1,\ldots,d$ and $j=1,\ldots,m$.
\end{definition}

\begin{definition}\label{def: kronecker}
	Let $A \in \R^{d_{1}\times m_{1}}$   and $B\in \R^{d_{2}\times m_{2}}$. The Kronecker product $(A\otimes B)$ is defined as the $d_1 d_2 \times m_1 m_2$ matrix  
	\begin{equation*}
	A\otimes B = \left( \begin{array}{ccc} A_{11} B & \cdots  & A_{1m_1} B \\ \vdots &  &  \vdots \\ A_{d_1 1} B & \cdots  & A_{d_1 m_1} B \end{array} \right).
	\end{equation*}
	or equivalently 
	$$ (A\otimes B)_{pq} = A_{ik} B_{jl}, \quad   p=(i-1)d_2 +j, \quad  q= (k-1)m_2 +l,$$
	for all $i=1,\ldots,d_1$, $j=1,\ldots,d_2$, $k=1,\ldots,m_1$ and $l=1,\dots,m_2$.
\end{definition}

{\begin{proposition} For matrices $A,B,C,D,X,w$ of suitable dimensions, the following relations hold:
		\begin{align}
		\vecop(AXB)&= \left( B^\top \otimes A\right) \vecop(X) \label{eq:kronecker1}\\
		\tr(A^\top w A)&=\vecop(A)^\top (I_m \otimes w) \vecop(A) \label{eq:kronecker2} \\
		\left( A \otimes B\right) \left( C \otimes D\right)&= \left( AC \otimes BD\right) \label{eq:kronecker3}\\
		\tr(A\otimes B)&=\tr(A)\tr(B) \label{eq:kronecker4}\\
			(A\otimes B)^{-1}&=(A^{-1}\otimes B^{-1}) \label{eq:kronecker5}\\
			\det (I_m \otimes A)&= \det(A)^m. \label{eq:kronecker6}
		\end{align}
\end{proposition}}

\section{Proof of Theorem~\ref{T:char2ZZ}}\label{A:proof}
Throughout this section we assume that the function $(s,u)\mapsto C_t(s,u)$ is continuous such that   \eqref{eq:boundC} holds, where  $C_t$ is  given by \eqref{eq:Ct}.

For each $t\leq T$, we consider the integral operator $\bold{C}_t$ induced by the kernel $C_t$
\begin{align*}
(\bold{C}_t f)(s) = \int_0^T C_t(s,u)f(u)du = \int_t^T C_t(s,u)f(u)du, \quad f \in L^2([0,T],\R^N), \quad s \leq T,
\end{align*}
where the last equality follows from the fact that $C_t(s,u)=0$ for any $u\leq t$.  We assume  that $t\mapsto \bold{C}_t$ is differentiable with  derivative $\dot{\bold{C}}_t$ given by \eqref{eq:diffC}.

\begin{lemma} Let $w \in \S^N_+$ and $t\mapsto R^w_{t,T}$ be defined as in \eqref{eq:resolventdef}. Then, 
	\begin{align}
		\sup_{t\leq T} \int_t^T  \int_t^T |R_t^w(s,u)|^2 dsdu&<\infty,\label{eq:boundR0}\\
	 \sup_{t\leq T} \sup_{t\leq s\leq T}   \int_t^T |R_t^w(s,u)|^2 du&<\infty,\label{eq:boundR1}\\
	\sup_{t\leq T} \sup_{t\leq s,u\leq T}  |R^w_{t,T}(s,u)|&<\infty.\label{eq:boundR}
	\end{align}
\end{lemma}

\begin{proof}
	Fix $t\leq T$. It follows from \eqref{eq:resolventdef} that 
	\begin{align*}
\int_t^T  \int_t^T |R_t^w(s,u)|^2 dsdu =  \sum_{n\geq 1} \frac{4(\lambda^n_{t,T})^2}{(1 + 2\lambda^n_{t,T} )^2} \leq 4\sum_{n\geq 1} {(\lambda^n_{t,T})^2} = 4 |w|\int_t^T  \int_t^T |C_t(s,u)|^2 dsdu,
	\end{align*}
	which, combined with \eqref{eq:boundC}, proves \eqref{eq:boundR0}.  Furthermore, an application of Jensen and Cauchy-Schwarz inequalities  on  the resolvent equation~\eqref{eq:resequation} yields
	\begin{align}\label{eq:resolventboundtemp}
	|R_t^w(s,u)|^2 \leq 8 \sup_{t'\leq T} \sup_{t'\leq s',u'\leq T}  |C_{t',T}(s',u')|^2 \left( 1 + T \int_t^T| R^w_t(z,u)|^2 dz\right), \quad t \leq s,u\leq T. 
	\end{align}
	 Integrating the previous identity with respect to $u$  leads to 
	\begin{align*}
	 \int_t^T |R_t^w(s,u)|^2 du \leq 8 T  \sup_{t'\leq T} \sup_{t'\leq s',u'\leq T}  |C_{t',T}(s',u')|^2 \left( 1 +  T\int_t^T\int_t^T| R^w_t(z,u)|^2 dz du\right),
	\end{align*}
	for all $s\geq t$. Combined with \eqref{eq:boundC} and \eqref{eq:boundR0}, we obtain  \eqref{eq:boundR1}. Finally,  it follows from the resolvent equation~\eqref{eq:resequation}  together with Jensen and Cauchy-Schwarz inequalities that 
	\begin{align*}
	|R_t(s,u)|^2 \leq  8 \sup_{t'\leq T} \sup_{t'\leq s',u'\leq T}  |C_{t',T}(s',u')|^2 \left( 1 +  T \int_t^T | R_t^w(s,z)|^2   dz\right)
	\end{align*}
	for all $t\leq s,u \leq T$. The right hand side is bounded by a finite quantity which does not depend on $t$, thanks to \eqref{eq:boundC} and \eqref{eq:boundR1}, yielding \eqref{eq:boundR}. 
\end{proof}

\begin{lemma}\label{L:continuity}
	For each $t\leq s\leq T$, $u\mapsto R^w_{t,T}(s,u)$ is continuous.
	For each $s,u\leq T$, $t\mapsto R^w_{t,T}(s,u)$ is continuous. 
\end{lemma}

\begin{proof}
	The first statement follows directly from the continuity of $(s,u)\mapsto C_t(s,u)$ for all $t\leq T$, the resolvent equation~\eqref{eq:resequation} and the dominated convergence theorem which is justified by \eqref{eq:boundC}. The second statement is proved as follows. Fix  $t\leq s,u\leq T$ and $h \in \R$ such that $0\leq t+h\leq T$.  The  resolvent equation~\eqref{eq:resequation} yields
	\begin{align*}
	R^w_{t+h}(s,u)-R^w_t(s,u) &= -2(C^w_{t+h}(s,u)-C^w_t(s,u))\\
	&\quad  - 2\int_t^T 	R^w_{t+h}(s,z)(C^w_{t+h}(z,u)-C^w_t(z,u))dz \\
	& \quad -2 \int_t^T (R^w_{t+h}(s,z) - R^w_{t}(s,z)) C^w_t(z,u) dz \\
	& \quad  +2\int_t^{t+h}	R^w_{t+h}(s,z) 	C^w_{t+h}(z,u)dz \\
	&=\bold{I}+\bold{II}+\bold{III}+\bold{IV}
	\end{align*}
	Since $t\mapsto C_t(s,u)$ is absolutely continuous,  we have that $\bold{I}\to 0$ as $h\to 0$ and also that $\bold{II}\to 0$ by an application of Cauchy--Schwarz inequality, the bound \eqref{eq:boundR}, and the dominated convergence theorem, which is justified by \eqref{eq:boundC}. To prove that $\bold{III}\to 0$, we fix  $q \in \R^N$ and $f_u(s):=C_t^w(s,u) q$. Then, 
	$$ \int_t^T  (R^w_{t+h}(s,z) - R^w_{t}(s,z)) C^w_t(z,u)q dz =  (\bold{R}^w_{t+h} f_u)(s) -(\bold{R}^w_{t} f_u)(s) \to 0,\quad \mbox{as } h\to 0,$$
    where the convergence follows from the continuity of $t\mapsto \bold{R}^w_t$ obtained  from that of $t\mapsto \bold{C}_t$, recall  \eqref{eq:boldRboldC}. By arbitrariness of $q$, we get $\bold{III}\to 0$.	Finally, it follows from \eqref{eq:boundC} and \eqref{eq:boundR}, that $\bold{IV}\to 0$ as $h\to 0$. Combining the above yields $	R^w_{t+h}(s,u)\to R^w_t(s,u)$ as $h\to 0$.
\end{proof}

\begin{lemma}\label{L:Rdiff}
$t\mapsto R^w_{t,T}(s,u)$ is absolutely continuous for almost every $(s,u)$ such that 
\begin{align*}
\dot{R}^w_{t,T}(s,u) &= -2  \sqrt{w}\dot{C}_{t,T}(s,u)\sqrt{w}  - 2 \int_t^T \sqrt{w}\dot{C}_{t,T}(s,z)\sqrt{w} R^w_{t,T}(z,u) dz \\
&\quad - 2 \int_t^T R^w_{t,T}(s,z)  \sqrt{w}\dot{C}_{t,T}(z,u)\sqrt{w} dz \\
&\quad -2  \int_t^T\int_t^T  R^w_{t,T}(s,z)  \sqrt{w}\dot{C}_{t,T}(z,z')\sqrt{w} R^w_{t,T}(z',u)   dzdz', \quad \mbox{on } [t,T]\; a.e.
\end{align*}
with the boundary condition
\begin{align}\label{eq:boundaryR}
R^w_{t,T}(\cdot,t) = R^w_{t,T}(t,\cdot)^\top = 0, \quad t \leq T. 
\end{align}
\end{lemma}

\begin{proof}
	The boundary condition \eqref{eq:boundaryR} follows from the resolvent equation~\eqref{eq:resequation} and the fact that $C_{t}(\cdot,t)=C_{t}(t,\cdot)^\top=0$, for all $t\leq T$. \\
\textit{{Step 1.}}
It follows from \eqref{eq:boldRboldC} and the fact that $t\mapsto \bold{C}_t$ is differentiable, that $t\mapsto \bold{R}^w_{t,T}$ is differentiable, so that 
\begin{align}\label{eq:Roperatordiff}
 (\bold{R}^w_{t+h,T} f)(s) =  (\bold{R}^w_{t,T} f)(s) + h  (\dot{\bold{R}}^w_{t,T} f)(s) + o(|h|), \quad f \in L^2([0,T],\R^{N}), \quad s\leq T,
\end{align}
for all $h\in \R$  such that $0\leq t+h\leq T$, with 
\begin{align*}
\dot{\bold{R}}^w_{t,T} = - 2 (\id +{\bold{R}}^w_{t,T}) \sqrt{w} \dot{\bold{C}}_{t,T}    \sqrt{w}   (\id +{\bold{R}}^w_{t,T}).  
\end{align*}
The right hand side being a composition of integral operators, $\dot{\bold{R}}^w_{t,T}$ is again an integral operator with kernel given by 
\begin{align*}
-2 (\delta +{{R}}^w_{t,T} ) \star \sqrt{w} \dot{C}_{t,T} \sqrt{w} \star(  \delta +{{R}}^w_{t,T} ),
\end{align*}
where by some abuse of notations $\delta$ denotes the kernel induced by the identity operator $\id$, that is $(\id f)(s)=\int_t^T \delta_{s=u}(ds,du)f(u)=f(s)$.\\
\textit{{Step 2.}}
Fix $f$ a measurable and bounded function, $t,h$ such that  $0\leq t+h\leq T$, $s\leq T$ and write 
\begin{align*}
({\bold{R}}^w_{t+h,T}f)(s,u)&= \int_{t+h}^T {{R}}^w_{t+h,T}(s,u)f(u) du \\
&= ({\bold{R}}^w_{t,T}f)(s,u) + \int_t^T \left( {{R}}^w_{t+h,T}(s,u) - {{R}}^w_{t,T}(s,u)   \right) f(u)du \\
&\quad - \int_t^{t+h} \left( {{R}}^w_{t+h,T}(s,u) - {{R}}^w_{t,T}(s,u)   \right) f(u)du  \\
&\quad + \int_t^{t+h} \left(  {{R}}^w_{t,T}(s,t) - {{R}}^w_{t,T}(s,u)   \right) f(u)du\\
&= \bold{I}+\bold{II}+\bold{III}+\bold{IV}
\end{align*}
where we used the vanishing boundary condition \eqref{eq:boundaryR} to introduce $R^w_{t,T}(s,t)$ in $\bold{IV}$.  Subtracting the previous equation to \eqref{eq:Roperatordiff} yields
\begin{align}\label{eq:II}
 \bold{II} = h (\dot{\bold{R}}^w_{t,T} f)(s) - \bold{III} - \bold{IV} + o(|h|).
\end{align}
An application of the Heine--Cantor theorem  yields that the continuity statements in Lemma~\ref{L:continuity} can be strengthened to   uniform continuity. Whence, for an arbitrary $\varepsilon >0 $ and for $h$ small enough,
$$\sup_{ u \in [t,t+h]}  |{{R}}^w_{t,T}(s,t) - {{R}}^w_{t,T}(s,u)|+   \sup_{ u \in [t,t+h]}  |{{R}}^w_{t+h,T}(s,u) - {{R}}^w_{t,T}(s,u)| \leq \varepsilon, \quad t\leq s\leq T. $$
This yields $|\bold{III}|+|\bold{IV}|\leq c h \varepsilon$, for some constant $c>0$, so that taking limits in \eqref{eq:II} gives
$$\lim_{h\to 0}\frac 1 h \bold{II}= (\dot{\bold{R}}^w_{t,T} f)(s).$$
An application of the dominated convergence theorem, which is justified by \eqref{eq:boundR}, yields that for any  $u,s\leq T$ $t\mapsto R_t(s,u)$ is absolutely continuous with
\begin{align*}
\dot {R}^w_t(s,u) =   -2 (\delta +{{R}}^w_{t,T} ) \star \sqrt{w} \dot{C}_{t,T} \sqrt{w} \star(  \delta +{{R}}^w_{t,T} ),
\end{align*}
which is the claimed expression. 
\end{proof}

We can now complete the proof of Theorem~\ref{T:char2ZZ}. 

\begin{proof}[Proof of Theorem~\ref{T:char2ZZ}] The claimed expression for the Laplace transform follows from \eqref{eq:charinfinite2}, the Riccati equation for $\Psi$ as defined in \eqref{eq:psi_tT} follows from Lemma~\ref{L:Rdiff},  and that of $\phi$ is straightforward from \eqref{eq:RiccatiopBigPhi}.
\end{proof}

\section*{Data availability statement}
No data were used to support this study.

 
\small

\bibliographystyle{plainnat}
\bibliography{bibl}
\end{document}